\documentclass{article}

\usepackage[utf8]{inputenc}
\usepackage[utf8]{inputenc}
\usepackage[utf8]{inputenc}
\usepackage{authblk}
\usepackage{times}
\usepackage[papersize={17cm,24cm},margin=22mm,top=17mm,headsep=5mm]{geometry}
\usepackage[dvipsnames]{xcolor}
\usepackage[english]{babel} % multilinguality

\definecolor{bred}{rgb}{0.8,0,0}
\usepackage{dsfont}
\usepackage{microtype}
\usepackage{graphicx}
\usepackage{subfigure}
\usepackage{booktabs} % for professional tables
\usepackage{amsfonts}
\usepackage{amsmath,amssymb,graphicx,epsfig}
\usepackage{enumerate,amsthm,epstopdf}

\usepackage[utf8]{inputenc} % allow utf-8 input
\usepackage[T1]{fontenc}    % use 8-bit T1 fonts

\usepackage{bbm}
\usepackage{xargs}
\usepackage{enumerate}
\usepackage{latexsym}
\usepackage{wasysym}

\usepackage{float}
\usepackage{hyperref}       % hyperlinks
\usepackage{url}            % simple URL typesetting
\usepackage{booktabs}       % professional-quality tables
\usepackage{amsfonts}       % blackboard math symbols
\usepackage{nicefrac}       % compact symbols for 1/2, etc.
\usepackage{microtype}      % microtypography
\usepackage[square,numbers]{natbib}
\usepackage{cleveref}
\usepackage[textsize=footnotesize]{todonotes}
\newcommand{\tn}{\bar{\theta}^\lambda_n}
\newcommand{\E}{\mathbb{E}}

\newcommand{\hl}{h_\lambda}

\newcommand{\pt}{\hat{\pi}_t}
\newcommand{\pkl}{\hat{\pi}_{k\lambda}}

\newcommand{\fin}{\phi^{-1}(z)}
\newcommand{\ptk}{\hat{\pi}_{\theta_t|\theta_{k\lambda}}}
\newcommand{\acu}{{AC\cup U}}
\newcommand{\aref}[1]{$\mathbf{A}$\ref{#1}}
\newcommand{\bref}[1]{$\mathbf{C}$\ref{#1}}
\newcommand{\lref}[1]{$\mathbf{B}$\ref{#1}}
\date{August 2022}

\begin{document}

\title{Taming under isoperimetry}
\author[1,4]{Iosif Lytras}
\author[1,2,3]{Sotirios Sabanis}

\affil[1]{\footnotesize{School of Mathematics, The University of Edinburgh, Edinburgh, UK}}
\affil[2]{ \footnotesize{The Alan Turing Institute, London, UK.}}
\affil[3]{ \footnotesize{ National Technical University of Athens, Athens, Greece.}}
\affil[4]{ \footnotesize{Archimedes/Athena RC, Athens, Greece}}
\date{November 2023}
\maketitle
\begin{abstract}
    In this article we propose a novel taming Langevin-based scheme called $\mathbf{sTULA}$ to sample from distributions with superlinearly growing log-gradient which also satisfy a Log-Sobolev inequality. We derive non-asymptotic convergence bounds in $KL$ and consequently total variation and Wasserstein-$2$ distance from the target measure. Non-asymptotic convergence guarantees are provided for the performance of the new algorithm as an optimizer. Finally, some theoretical results on isoperimertic inequalities for distributions with superlinearly growing gradients are provided. Key findings are a Log-Sobolev inequality with constant independent of the dimension, in the presence of a higher order regularization and a Poincare inequality with constant independent of temperature and dimension under a novel non-convex theoretical framework.
\end{abstract}

\section{Introduction}
We consider a non-convex stochastic optimization problem
$$
\operatorname{minimize} u(\theta):=\mathbb{E}[f(\theta, X)]
$$
where $\theta \in \mathbb{R}^d$ and $X$ is a random element. We aim to build an estimate $\hat{\theta}$ such that the expected excess risk $\mathbb{E}[u(\hat{\theta})]-\inf _{\theta \in \mathbb{R}^d} u(\theta)$ is minimized.
It is well known that for large value of $\beta$, the measure $\pi_\beta(x) = \frac{e^{-\beta u(x) }}{\int_{\mathbb{R}^d}e^{-\beta u(x)}dx}$ is concentrated around the minimizers of $u$ therefore, if one decomposes the excess risk problem as 
\begin{equation}\label{eq-optimization problem}
    \mathbb{E}[u(\hat{\theta})]-\inf _{\theta \in \mathbb{R}^d} u(\theta)=\mathbb{E}[\underbrace{u(\hat{\theta})]-\E[u(\theta_\infty)]}_{T_1}+\underbrace{\E[u(\theta_\infty)]-\inf _{\theta \in \mathbb{R}^d} u(\theta)}_{T_2}
\end{equation}
where $\theta_\infty$ is distributed according to $\pi_\beta$.
It is well known that,  $T_2$ becomes small when the temperature parameter $\beta>0$ is large, and thus the main effort is in the direction of  minimizing $T_1$.
This can be achieved by building a chain based on a Langevin sampling algorithm to sample from the distribution $\pi_\beta.$ The non-asymptotic sampling behaviour of this algorithm is the main focus of this article.\\
Sampling from a high dimensional distribution using Langevin-based algorithms has been a topic of interest in many fields such as Bayesian statistics and machine learning.
The Langevin-based sampling relies on the notion that under mild conditions, the Langevin SDE
\begin{equation}
 X_0=\theta_0, \quad   dX_t=-\nabla u(X_t) +\sqrt{\frac{2}{\beta}}dB_t
\end{equation}
admits $\pi_\beta$ as an invariant measure.
One popular approach is to consider the Unadjusted Langevin Algorithm (ULA) which corresponds to the respective Euler-Maruyama discretization scheme given as
\[\theta^{ULA}_0=\theta_0, \quad \theta^{ULA,\lambda}_{n+1}=\theta^{ULA,\lambda}_{n}-\lambda h(\theta^{ULA,\lambda}_{n}) + \sqrt{\frac{2\lambda}{\beta}}\xi_{n+1}\]
where $\{\xi\}_n$ is a sequence of $d$-dimensional Gaussian random variables and $\lambda>0$ is the step-size of the algorithm and $h:=\nabla u$.
There have been a lot of work in providing non-asymptotic results for ULA under various assumptions such as Lipschitz continuity of $h$ and convexity of $u$. Under the assumption of convexity and gradient Lispchitz continuity important results are obtained in \citet{dalalyan2017theoretical, durmus2017nonasymptotic, durmus2019high, hola, convex}, while in the non-convex case, under convexity at infinity or dissipativity assumptions, one may consult \citet{berkeley} \citet{majka2020nonasymptotic}, \citet{erdogdu2022convergence} for ULA while for the Stochastic Gradient variant (SGLD) important works are \citet{raginsky},\citet{nonconvex} \citet{zhang2023nonasymptotic}.\\ More recently, starting with the work of \citet{vempala2019rapid} important estimates have been obtained under the assumption that the target measure $\pi_\beta$ satisfies an isoperimetric inequality and the gradient of $u$ satisfies a global Lipschitz continuity, \citet{mou2022improved}, \citet{ balasubramanian2022towards}.
The latter assumption has also been relaxed to a weakly smooth assumption (with the gradient still satisfying a linear growth property), see \citet{nguyen2021unadjusted} and \citet{erdogdu2021convergence}.\\
For gradients satisfying a superlinear growth condition different techniques need to be explored. The reason for this is when the drift coefficient has superlinear growth, the Euler Marauyama scheme (which is the basis for ULA) diverges in the strong sense. That corresponding result can be found in \citet{hutzenthaler2011} where it is proven that the difference of the exact solution of the corresponding stochastic differential equation (SDE) and of the numerical approximation at even a finite time point, diverges to infinity in the strong mean square sense. This inspired the use of taming technology for the approximation of such SDEs which was introduced in \citet{hutzenthaler2012} and subsequently by using the so-called Euler-Krylov approximations in  \citet{tamed-euler,SabanisAoAP} to address this issue. Naturally, this has led to the development of tamed Langevin-based sampling algorithms in \citet{tula}, \citet{hola}, \citet{johnston2023kinetic} under a  strong convexity assumption and in the non-convex setting in \citet{TUSLA}, \citet{lim2021polygonal} and \citet{neufeld2022non}.\\
In this article we propose a novel taming scheme called $\mathbf{sTULA}$ and provide non-asymptotic convergence bounds in $KL$ and in total variation and Wasserstein-$2$ distance from the target measure $\pi_\beta$, assuming polynomial local Lipschitz continuity for the gradient and a Log-Sobolev inequality for $\pi_\beta.$ Using this result, we produce non-asymptotic guarantees for the solution of the excess risk optimization problem. Finally, we provide some new results where Poincare or Log-Sobolev inequalities are derived under novel theoretical frameworks. More specifically, Theorem \ref{lemma-comparisonying} offers additional insight in cases where a certain convexity at infinity condition is met, where the gradient is allowed to grow polynomially at infinity. In particular, Corollary \ref{cor-reg} deals with the cases where high order regularization is added, which is quite important in many practical applications such as the fine tuning of Neural Networks (see \cite{TUSLA}, \cite{lim2021non}), deriving Log-Sobolev inequality with constant independent of the dimension.\\ Another direction explored is the derivation of a novel theoretical framework beyond any convexity assumption, producing a Poincare inequality with constant which doesn't depend explicitly on the temperature and dimension, see Theorem \ref{Poincare constant general} and under additional assumptions, a Log Sobolev inequality with constant depending polynomially on the temperature parameter and the dimension. These new results along with the novel techniques used could possibly enhance the undertanding of isoperimetric inequalities in different scenarios stemming from practical applications and pave the way for many interesting findings.\\ 
\textbf{Notation.} We conclude this Section by introducing some notation. The Euclidean norm of a vector $b \in \mathbb{R}^d$, the spectral norm and the Frobenius norm of a matrix $\sigma \in \mathbb{R}^{d \times m}$ are denoted by $|b|,||A||$ and $||A||_{\mathrm{F}}$ respectively. $A^{\top}$ is the transpose matrix of $A$. Let $f: \mathbb{R}^d \rightarrow \mathbb{R}$ be a twice continuously differentiable function. Denote by $\nabla f, \nabla^2 f$ and $\Delta f$ the gradient of $f$, the Hessian of $f$ and the Laplacian of $f$ respectively. We denote the $i-th$ order Jacobian of an $i$ times differentiable function $f: \mathbb{R}^d \rightarrow \mathbb{R}$ as $J^{(i)}(f).$ We also denote $\mathcal{H}^k$ the usual Sobolev space.
$\|\cdot\|_V$ is the total variation denoted by $\|\cdot\|_{T V}$. Let $\mu$ and $\nu$ be two probability measures on a state space $\Omega$ with a given $\sigma$-algebra. If $\mu \ll \nu$, we denote by $d \mu / d \nu$ the Radon-Nikodym derivative of $\mu$ w.r.t. $\nu$. Then, the Kullback-Leibler divergence of $\mu$ w.r.t. $\nu$ is given by
$$
{H}_\nu(\mu)=\int_{\Omega} \frac{d \mu}{d \nu} \log \left(\frac{d \mu}{d \nu}\right) d \nu .
$$
We say that $\zeta$ is a transference plan of $\mu$ and $\nu$ if it is a probability measure on\\ $\left(\mathbb{R}^d \times \mathbb{R}^d, \mathcal{B}\left(\mathbb{R}^d\right) \times \mathcal{B}\left(\mathbb{R}^d\right)\right)$ such that for any Borel set $A$ of $\mathbb{R}^d, \zeta\left(A \times \mathbb{R}^d\right)=\mu(A)$
and $\zeta\left(\mathbb{R}^d \times A\right)=\nu(A)$. We denote by $\Pi(\mu, \nu)$ the set of transference plans of $\mu$ and $\nu$. Furthermore, we say that a couple of $\mathbb{R}^d$-valued random variables $(X, Y)$ is a coupling of $\mu$ and $\nu$ if there exists $\zeta \in \Pi(\mu, \nu)$ such that $(X, Y)$ is distributed according to $\zeta$. For two probability measures $\mu$ and $\nu$, the Wasserstein distance of order $p \geq 1$ is defined as
$$
W_p(\mu, \nu)=\left(\inf _{\zeta \in \Pi(\mu, \nu)} \int_{\mathbb{R}^d \times \mathbb{R}^d}|x-y|^p d \zeta(x, y)\right)^{1 / p}.
$$

\section{Theoretical Framework}
In the following definitions we assume that $\pi(x):=\pi_\beta(x) = \frac{e^{-\beta u(x) }}{\int_{\mathbb{R}^d}e^{-\beta u(x)}dx}$ is a Gibbs measure associated with the Langevin SDE
\begin{equation}
    dX_t=-\nabla u(X_t)dt +\sqrt{\frac{2}{\beta}}dB_t.
\end{equation}
The generator of the SDE is given by $Lf=\frac{1}{\beta}\Delta f -\langle \nabla f,\nabla u\rangle  $, $f \in \mathcal{H}^2$ and the carr\'{e} du champ operator $\Gamma(f,g)=\frac{1}{\beta}\langle \nabla f,\nabla g\rangle $, $f,g\in \mathcal{H}^1$
\subsection{Definitions and key isoperimetric theorems}
\newtheorem{Def1}{Definition}[section]
\begin{Def1}[Poincaré Inequality] \label{Poincarédef}
 A Gibbs probability measure $\pi_\beta$  satisfies the Poincar\'{e} inequality with constant $\varrho>0$, if for all test functions $f \in H^1(\mu)$
$$
(\mathrm{PI}(\varrho)) \quad \operatorname{var}_{\pi_\beta}(f):=\int_{\mathbb{R}^d}\left(f-\int_{\mathbb{R}^d} f \mathrm{~d} \pi_\beta\right)^2 \mathrm{~d} \mu \leq \frac{1}{\varrho} \frac{1}{\beta} \int|\nabla f|^2 \mathrm{~d} \pi_\beta
$$
\end{Def1}
\newtheorem{remark}[Def1]{Remark}
\newtheorem{ass}{A}
\newtheorem{assB}{B}
\newtheorem{assC}{C}
\newtheorem{theorem}[Def1]{Theorem}
\newtheorem{lemma}[Def1]{Lemma}
\newtheorem{corollary}[Def1]{Corollary}
\newtheorem{Def}[Def1]{Definition}
\newtheorem{proposition}[Def1]{Proposition}

\begin{Def}[LSI] \label{LSIdef}
A Gibbs probability measure $\pi_\beta$ satisfies the logarithmic Sobolev inequality with constant $\alpha>0$, denoted $\operatorname{LSI}(\alpha)$, if for all probability measures $\nu$ such that $\nu \ll \pi_\beta$ $f:=\frac{d \nu}{d \pi_\beta} \in H^1(\mathbb{R}^d)$
$$
(\mathrm{LSI(\alpha)}) \quad
H_{\pi_\beta}(\nu):=\int_{\mathbb{R}^d} f\log f d \pi_\beta \leq \frac{1}{2 \alpha} \int_{\mathbb{R}^d} \frac{\Gamma(f,f)}{f} d \pi_\beta:= \frac{1}{2 \alpha} I_{\pi_\beta}(\nu),
$$
\end{Def}
\begin{Def}[Talagrand inequality]
    A probability measure $\mu$ is then said to satisfy the transportation-entropy inequality ( Talagrand) $W_2 H(C)$ where $C>0$ is some constant, if for all probability measure $\nu$
$$ (W_2 H(C))\quad
W_2(\nu, \mu) \leq \sqrt{2 C H_\mu(\nu)}
$$
\end{Def}
\begin{theorem}[Otto-Villani]\label{theo-Talagrand}
    Suppose that a probability measure $\mu$ satisfies LSI($a$). Then, it also satisfies a Talagrand inequality with the same constant i.e $W_2H(\frac{1}{a})$
\end{theorem}
\begin{theorem}[Bakry-Emery criterion, \citet{bakry2006diffusions}]\label{Bakry-Emery}
Let $U: D \rightarrow \mathbb{R}$ be a potential associated with Gibbs measure and normalizing constant $Z_\mu$,
$$
\mu_\beta(\mathrm{d} x)=Z_\mu^{-1} \exp \left(-\beta U(x)\right) \mathrm{d} x
$$
on a convex domain $D\subset \mathbb{R}^d$ and assume that there exists $m>0$ such that $\nabla^2 U(x) \geq m I_d$ for all $x \in \mathbb{R}^n$. Then $\mu$ satisfies $\mathrm{PI}(\varrho)$ and $\operatorname{LSI}(\alpha)$ with
$$
\varrho \geq m \quad \text { and } \quad \alpha \geq m.
$$
\end{theorem}
%\begin{theorem}[Holley-Stroock perturbation principle]
% Let $H$ be a Hamiltonian with Gibbs measure $\mu({d} x)=Z_\mu^{-1} \exp \left(-\beta H(x)\right) {d} x$. Further, let $\tilde{H}$ denote a bounded perturbation of $H$ and let $\tilde{\mu}_{\beta}$ denote the Gibbs measure associated to the Hamiltonian $\tilde{H}$. If $\mu$ satisfies $\mathrm{PI}(\varrho)$ or $\mathrm{SI}(\alpha)$ then also $\tilde{\mu}$ satisfy $\mathrm{PI}(\tilde{\varrho})$ or $\operatorname{LSI}(\tilde{\alpha})$ respectively, where the constants satisfy the bounds
% $$
% \tilde{\varrho} \geq \exp \left(-\beta \operatorname{osc} \psi\right) \varrho \quad \text { and } \quad \tilde{\alpha} \geq \exp \left(-\beta \operatorname{osc}(H-\tilde{H})\right) \alpha,
% $$
% where $\operatorname{osc}(H-\tilde{H}):=\sup (H-\tilde{H})-\inf (H-\tilde{H})$
% \end{theorem}
\subsection{Assumptions}\label{Assumptions}
 % $u\in \mathcal{C}^{\frac{d}{2}+3}(\mathbb{R}^d)$ 
Let $u\in \mathcal{C}^4(\mathbb{R}^d) $ and $h:=\nabla u.$
The following set of Assumptions shapes the necessary framework to obtain our results.
\begin{ass}\label{ass-derivbound}
There exists $l>0$ such that \[\max\{|h(x)|,||J^{(i)}(h)(x)||\}\leq L(1+|x|^{2l}),\quad \forall x\in \mathbb{R}^d, \quad  i=1,2,3.\]
\end{ass}
\begin{ass}(Polynomial Lipschitz continuity)\label{ass-pol lip}
There exist $L', l' >0$  such that
\[|h(x)-h(y)|\leq L'(1+|x|+|y|)^{l'}|x-y| \quad \forall x,y \in \mathbb{R}^d.\]
\end{ass}
\begin{ass}(2-Dissipativity)\label{ass-2dissip}
There exist $a, b>0$ such that \[\langle h(x),x\rangle \geq a|x|^2-b \quad \forall x\in \mathbb{R}^d.\]
\end{ass}

\begin{ass} \label{ass-initial cond}
Let $\hat{\pi}_0$ be the initial distribution of the algorithm. Then,
$\hat{\pi}_0$ has exponential decay, $|\nabla \log \hat{\pi}_0|$ has polynomial growth and $||\nabla ^2 \log \hat{\pi}_0||$ has polynomial growth.
\end{ass}
\begin{assB}\label{ass-LSI}
The measure $\pi_\beta$ satisfies a Log-Sobolev inequality with a constant $C_{LSI}$.
\end{assB}
\newtheorem{remarkforb1}[Def1]{Remark}
\begin{remark}
Assumption \lref{ass-LSI} is a quite general assumptions and applies to many different scenarios. In a general non-convex setting this constant may have an exponential dependence on the dimension and $\beta.$
As shown in Theorem \ref{theo-LSI constant}, in order obtain an LSI with polynomial dependence on the temperature and dimension,  one can replace assumption \lref{ass-LSI} by the following assumptions:
\end{remark}

\begin{assC}\label{ass11}
There exists $K>0$ such that \[\nabla^2 u(x)\geq -KI_d \quad \forall x\in \mathbb{R}^d.\] where $I_d$ is the $d\times d$ identity matrix.
\end{assC}
\begin{assC} \label{ass-deltau}
    There exists $C'>0$ such that,
    \[||\nabla^2 u(x)||\leq C'(1+|h(x)|) \quad \forall x\in \mathbb{R}^d.\]
    This also implies
    \[\Delta u \leq 2 C'd(1+|h(x)|^2)\quad \forall x\in \mathbb{R}^d.\]
\end{assC}

\begin{assC}\label{ass-unique minimum}
$u$ has a unique local minimum.
\end{assC}
\begin{assC} \label{ass-Morse}
Let $\{\bar{\lambda}_i(\nabla^2 u)\}_{i=1}^d$ be the eigenvalues of the matrix $\nabla^2 u$.
There exists a constant $l^*>0$ such that
\[l^* \leq \inf \left\{\left|\bar{\lambda}_i\left(\nabla^2 u(y)\right)\right| \mid h(y)=0, i \in \{1,2,...,d\}\right\}.\]
Denoting by $S$ the set containing all saddle points and local maxima of $u$,
\[\sup_{x\in S} \bar{\lambda}_{min}(\nabla^2 u(x))\leq -l^*<0.\]
\end{assC}
\begin{remark}\label{remark about LSI}
    It is important to note that a Log-Sobolev inequality can  be deduced only from assumptions \bref{ass11},\aref{ass-2dissip} albeit with an exponential dependence on the temperature and the dimension.
\end{remark}

\section{Discussion about Log-Sobolev inequality}
LSI is widely used assumption for the target distribution of interest in the field of Langevin sampling since it implies concentration of measure and sub-Gaussian tails (\citet{Ledoux}). It was initially proved by Gross (\citet{gross1975logarithmic}) for the Gaussian measure and then extended by the Bakry-Emery theorem ( Theorem \ref{Bakry-Emery}) to the logconcave case. It was also extended to bounded petrubations via the Hooley-Strook petrubation theorem (\citet{holley1986logarithmic}). Moreover, it is preserved under contractions.
\\ Recently in the sampling temperature free (i.e setting $\beta=1$) literature, there have been numerous works
where isoperimetric inequalities are used to prove non-asymptotic convergence of a sampling algorithm to the target measure. In the works of \citet{vempala2019rapid} and \citet{mou2022improved} a differential inequality regarding $KL$ divergence between the sampling algorithm and the target measure was established under $LSI$. There have been a lot of important works towards this direction or under weaker functional inequalities such as modified Log Sobolev and Poincaré (\citet{chewi2021analysis}, \citet{nguyen2021unadjusted}, \citet{erdogdu2021convergence}).
% For example, the work of \citet{nguyen2021unadjusted} using a nice technique of convexification of a non convex domain assuming weak smoothness (i.e gradient with linear growth but less than global Lipschitz continuous),  positive semidefinite Hessian outside of a ball, and 2-dissipativity .\\
When one uses sampling algorithms to solve the excess risk optimization problem, the presence of the large temperature $\beta$ adds another layer of difficulty, as it is increasingly complicated to derive a Log-Sobolev constant with non-exponential dependence on the temperature and the dimension. In the important work in \citet{menz2014poincare} a dimension-free with exponential dependence on $\beta$ Log Sobolev constant was obtained under general assumptions. In fact, the result was also shown to be optimal for a specific one-dimensional example (see Section 2.4 in \citet{menz2014poincare}).\\
Under more restrictive but realistic assumptions, a result establishing Log-Sobolev constants with polynomial dependence in the temperature and the dimension has only be derived in \citet{li2020riemannian} in the compact manifold setting.\\
The interested reader may wish to be informed by an excellent analysis of the use of isoperimetric inequalities in the context of sampling and optimization in \citet{raginsky}, where an exponential in both temperature and the dimension Log Sobolev constant is obtained under Assumptions similar to Assumptions \bref{ass11} and \aref{ass-2dissip}.
\section{Presentation of the algorithm}\label{presentation tamed}
We propose a new Polygonal Euler-Krylov (tamed) scheme which is inspired by the construction developed in the work of \citet{johnston2024strongly} and \citet{johnston2023kinetic}. The novelty of this scheme lies in the fact that it preserves the dissipativity condition (Assumption \aref{ass-2dissip}) of the initial gradient (see Lemma \ref{dissipativity-tamed}) . This enables the derivation of moment bounds without additional assumptions such as Assumption $\mathbf{H}2 ii)$ in \citet{tula}, as  the algorithm satisfies the important condition 3 as described in \citet{lim2021polygonal}.\\
Our algorithm is given by the following iterative scheme.
Let $\lambda>0$ the stepsize of the algorithm and let
\[f(x)=h(x)-ax.\] We define $f_\lambda (x)=\frac{f(x)}{1+\sqrt{\lambda}|x|^{2l}}.$
We propose the splitted tamed unadjusted Langevin algorithm ($\mathbf{sTULA}$)
\begin{equation}\label{eq-sTULA}
    \bar{\theta_0}^\lambda:=\theta_0, \quad \bar{\theta}_{n+1}^\lambda:=\bar{\theta}_n^\lambda-\lambda h_\lambda\left(\bar{\theta}_n^\lambda\right)+\sqrt{2 \lambda \beta^{-1}} \xi_{n+1}, \quad n \in \mathbb{N}_0,
\end{equation}
where $\lambda>0$ is the step size, $\left(\xi_n\right)_{n \in \mathbb{N}}$ is a sequence of independent standard $d$-dimensional random variables independent of the $\mathbb{R}^d$-valued random variable $\theta_0$, and where for all $\theta \in \mathbb{R}^d$,
the tamed coefficient of our iterative scheme shall be given as
\begin{equation}
h_\lambda(x)= ax + f_\lambda (x).
\end{equation}
The restrictions on the stepsize will be given by the respective ones the moment bounds (i.e Lemma \ref{2moments} and \ref{highmom}) and some important lemmas in the Appendix (i.e Lemma \ref{pt-decay}).
We set $\lambda_{max}=\min\{1,\frac{1}{4(2a +4L)^2}\}.$
\section{Main results}
First we present a result which shows that under additional assumptions, for $\beta$ large enough one may obtain a Log Sobolev constant with at most polynomial dependence on the dimension and temperature.
\newtheorem{Poincare constant}[Def1]{Theorem}
\begin{Poincare constant}\label{Poincare constant general}
Let Assumptions \aref{ass-derivbound}, \aref{ass-pol lip} and \bref{ass-deltau}-\bref{ass-Morse} hold. In addition, we assume that \begin{equation}\label{eq-limitcond}
\exists R>0, \quad c_H>0: \quad |h(x)|\geq c_H \quad \forall |x|\geq R. 
\end{equation} 

Then, for $\beta \geq \mathcal{O}(d^6)$,
$\pi_\beta$ satisfies a Poincare inequality with constant $(C_P)^{-1}$ which doesn't depend explicitly on $d$ and $\beta$.
\end{Poincare constant}
\newtheorem{LSI constant}[Def1]{Corollary}
\begin{LSI constant}\label{theo-LSI constant}
Let Assumptions \aref{ass-derivbound}-\aref{ass-initial cond} and \bref{ass11}-\bref{ass-Morse}  hold.
Then, for $\beta \geq \mathcal{O}(d^6)$,
$\pi_\beta$ satisfies \lref{ass-LSI} with constant $C_{LSI}$ such that \[(C_{LSI})^{-1}\leq \mathcal{O}(\beta^2).\]
\end{LSI constant}
We also present a result where Log Sobelev inequality is obtained under a " strongly convex at infinity" condition.
\begin{theorem}\label{lemma-comparisonying}
Suppose $u:\mathbb{R}^d\rightarrow \mathbb{R},$ $ u\in \mathcal{C}^2$ satisfying \aref{ass-pol lip} and \begin{equation}\label{eq-convinf}
    \langle \nabla u(x)-\nabla u(y),x-y\rangle \geq \left(c_1(|x|^{2r}+|y|^{2r})-c_2(|x|^l+|y|^l)-c_3\right)|x-y|^2 \quad \forall x,y \in \mathbb{R}^d,
\end{equation} for some $c_1,c_2,c_3>0$ and $2r>l>0$. \\
    Then, $\pi_\beta$ satisfies an LSI with constant independent of the dimension and exponential in $\beta$.
\end{theorem}
Using the previous result an LSI can be deduces for functions with locally Lipschitz gradients with added high-order regularization.
\begin{corollary}\label{cor-reg}
     Let $g:\mathbb{R}^d\rightarrow\mathbb{R}$, $g\in\mathcal{C}^2$  such that \[|\nabla g(x)-\nabla g(y)|\leq L (1+|x|^l+|y|^l)|x-y|\quad \forall x,y\in \mathbb{R}^d.\] Let $r>\frac{l}{2}>0$ and $\eta>0$. Let $u:=\eta |\cdot|^{2r+2}+g. $ \\Then, $\pi_\beta:=\frac{e^{-\beta u(x)}}{\int_{\mathbb{R}^d}e^{-\beta u(x)}}$ satisfies an LSI with constant independent of the dimension and exponential in $\beta$.
\end{corollary}
Using the tamed scheme presented in Section \ref{presentation tamed} we reach the following non-asymptotic results.
\newtheorem{rate}[Def1]{Theorem}
\begin{rate}\label{H-rate}
Let Assumption \aref{ass-derivbound}-\aref{ass-initial cond} hold. In addition, we assume \lref{ass-LSI}.
Let $\rho_n$ be the distribution of $n-th$ iterate of the algorithm \eqref{eq-sTULA}.
Then,  for $\lambda\leq \lambda_{max}$,
\[H_{\pi_\beta}(\rho_n)\leq e^{-\frac{3}{2} C_{LSI}\lambda (n-1) } H_{\pi_\beta}(\rho_0) + \frac{\beta \hat{C}}{\frac{3}{2} C_{LSI}}\lambda \]

where $\hat{C}$ depends polynomially on the dimension.

As a result, to achieve an accuracy $H_{\pi_\beta}(\rho_n)\leq \epsilon$, for $\lambda\leq \frac{3 \epsilon C_{LSI}}{2 \beta \hat{C}}$, one needs
\[\left\{
	\begin{array}{ll}
		n\geq \frac{2 \beta \hat{C}}{\epsilon}C_{LSI}^{-1}\log(\frac{2}{\epsilon}H_{\pi_\beta}(\rho_0)) \quad \text{iterations}  & \mbox{under \lref{ass-LSI} }  \\
		n\geq \mathcal{O}\left(\frac{2}{\epsilon}\log(\frac{2}{\epsilon}){poly}(d)\right) \quad \text{iterations} & \mbox{if} \quad \text{ \lref{ass-LSI} is replaced by \bref{ass11}-\bref{ass-Morse}} \quad \\& \text{and}  \quad \beta \geq \mathcal{O}(d^6)
  \\n\geq \mathcal{O}\left(\frac{2}{\epsilon}\log(\frac{2}{\epsilon}){poly}(d)\right) \quad \text{iterations}& \mbox{if} \quad \text{\lref{ass-LSI} is replaced by \eqref{eq-convinf}}
	\end{array}
\right.\]
\end{rate}
Using Pinsker's inequality one obtains non-asymptotic bounds in total variation distance.
\newtheorem{TV rate}[Def1]{Corollary}
\begin{TV rate}
Let Assumption \aref{ass-derivbound}-\aref{ass-initial cond} hold. In addition, we assume  \lref{ass-LSI}.
Let $\rho_n$ be the distribution of $n-th$ iterate of the algorithm.
Then, for $\lambda\leq \lambda_{max}$
\[||\mathcal{L}(\Bar{\theta}^\lambda_n)-\pi_\beta||_{TV}\leq \frac{\sqrt{2}}{2}\sqrt{e^{-\frac{3}{2}(C_{LSI}) \lambda (n-1) } H_{\pi_\beta}(\rho_0) + \frac{\beta \hat{C}}{\frac{3}{2}(C_{LSI})}\lambda}.\]
As a result, to achieve an accuracy $||\mathcal{L}(\Bar{\theta}^\lambda_n)-\pi_\beta||_{TV}\leq \epsilon$, for $\lambda\leq \frac{3 \epsilon^2 C_{LSI}}{2 \beta \hat{C}}$ one needs $n\geq \left(\frac{2 \beta \hat{C}}{\epsilon^2}C_{LSI}^{-1}\right)\log(\frac{2}{\epsilon}H_{\pi_\beta}(\rho_0))$ iterations.
If \lref{ass-LSI} is replaced by \bref{ass11}-\bref{ass-Morse} and $\beta \geq \mathcal{O}(d^6)$, or \lref{ass-LSI} is replaced by \eqref{eq-convinf}, one achieves the accuracy in 
$\mathcal{O}\left(\frac{2}{\epsilon^2}\log(\frac{2}{\epsilon}){poly}(d)\right) $ iterations.
\end{TV rate}
Using Talagrand's inequality one deduces the following result regarding the convergence in Wasserstein distance. 
\newtheorem{W2 rate}[Def1]{Corollary}
\begin{W2 rate}\label{W2 rate}
Let Assumption \aref{ass-derivbound}-\aref{ass-initial cond} hold. In addition, we assume \lref{ass-LSI}.
Then, there holds,
\[W_2(\mathcal{L}(\bar{\theta}^\lambda_n),\pi_\beta)\leq \frac{\sqrt{2}}{\sqrt{C_{LSI}}}\left( e^{- \frac{3}{4}(C_{LSI}\lambda(n-1)} H_{\pi_\beta}(\rho_0) + \sqrt{\frac{\beta \hat{C}}{\frac{3}{2}(C_{LSI})}\lambda}\right).\]
As a result, to achieve an accuracy $W_2(\mathcal{L}(\bar{\theta}^\lambda_n),\pi_\beta)\leq \epsilon$, for $\lambda\leq \frac{3 \epsilon^2 C_{LSI}^2}{2 \beta \hat{C}}$ one needs $n\geq \left(\frac{2 \beta \hat{C}}{\epsilon^2}C_{LSI}^{-2}\right)\log(\frac{2}{\epsilon}H_{\pi_\beta}(\rho_0))$ iterations.
If \lref{ass-LSI} is replaced by \bref{ass11}-\bref{ass-Morse} and $\beta \geq \mathcal{O}(d^6)$, or \lref{ass-LSI} is replaced by \eqref{eq-convinf}, one achieves the accuracy in 
$\mathcal{O}\left(\frac{2}{\epsilon^2}\log(\frac{2}{\epsilon}){poly}(d)\right) $ iterations.
\end{W2 rate}
One can use the previous result to provide non-asymptotic bounds for the excess risk optimization problem.
\newtheorem{excess risk}[Def1]{Corollary}
\begin{excess risk}[\citet{TUSLA}, Theorem 2, Adapted ]\label{cor-excess risk}
Let Assumptions \aref{ass-pol lip}, \aref{ass-2dissip} hold.
Let  $u_*=min_{x\in \mathbb{R}^d} u(x)$.
Then,
\[\mathbb{E}\left[u\left(\bar{\theta}_n^\lambda\right)\right]-u_{\star} \leq C W_2\left(\mathcal{L}\left(\bar{\theta}_n^\lambda\right), \pi_\beta\right) + \mathcal{O}\left(\frac{\log \beta}{\beta}\right).\]
\end{excess risk}
\subsection{Main contributions and comparison with relevant literature}
The main contribution of this article is to provide non-asymptotic bounds in $KL$-divergence, $W_2$ and $TV$ using isoperimetric techniques
for sampling from distributions with non-convex objective functions  with superlinearly growing gradients, under the framework shaped specifically by Assumption \ref{ass-pol lip}. These non-asymptotic results are used to provide bounds for the excess risk optimization problem.\\
One technical novelty of this article lies in the use of this new taming scheme which inherits the dissipativity condition of the original gradient. The non-asymptotic convergence results for our scheme are achieved by establishing under our optimization framework (where $\beta$ is large )  a differential inequality of relative entropy analogous to the one proved in \citet{vempala2019rapid} without a global gradient Lipschitz assumption. 
\\ Another important contribution which is also of technical interest is the extension of the techniques of \citet{li2020riemannian} by proving a Log Sobolev inequality under our unconstrained optimization framework, showing that the Log Sobolev constant has at most polynomial dependence on the dimension and temperature.\\
Since the prevailing approach in the literature, under the assumption of local (instead of global) gradient Lipschitz continuity use taming variants of Langevin algorithms, it is natural to compare our results with
\citet{neufeld2022non}, \citet{tula}.\\ An additional important contribution is the proof of an Log-Sobolev inequality with independent of the dimension (but exponential in temperature) Log-Sobolev constant, under convexity infinity assumption though allowing for superlinearly growing gradient. This result is an extension of the work of Proposition 2 in \citet{ma2019sampling}, beyond the global Lipschitz gradient case.
This article can be seen as an extension of the work of \citet{tula} in the non-convex case. In the current work non-asymptotic total variation bounds are obtained with the same $\mathcal{O}(\lambda)$ rate as the ones in \citet{tula} . In addition, contrary to the increased disspativity assumption in H2 ii) of \citet{tula} we are able to obtain the results using only $2$- dissipativity, owning to a more sophisticated scheme which inherits the required dissipativity properties (see Lemma \ref{dissipativity-tamed}).\\
Compared with the work in \citet{neufeld2022non}, since the results are proved without the use of a contraction semi-metric result of \citet{Harris}, under \aref{ass-derivbound}-\aref{ass-initial cond} and \bref{ass11}-\bref{ass-Morse} we are able to derive the same bounds with respect to the step-size but with a better control on the dimension and $\beta$ (polynomial vs exponential in \citet{neufeld2022non}). It should be noted that, as shown in Theorem \ref{lemma-comparisonying}, Assumption \aref{ass-pol lip} and Assumption 3i) of \citet{neufeld2022non} implies  \lref{ass-LSI} with constant independent of the dimension. Thus, using only Assumptions \aref{ass-derivbound}-\aref{ass-initial cond} and Assumption 3i) of \citet{neufeld2022non}, the constants in Corollary \ref{W2 rate} are polynomial in the dimension with the added benefit that  non-asymptotic bounds can also be derived for both $KL$ and total variation distance.\\
%Finally, when compared with the work in \citet{TUSLA} we are able to obtain better non-asymptotic bounds in $W_2$ distance both in stepsize ( rate $\frac{1}{2}$ vs $\frac{1}{4}$ in \citet{TUSLA}) and in the constants dependence of the dimension and temperature (polynomial vs exponential in \citet{TUSLA}).
%However, it is important to note that this comparison is not absolutely fair as in \citet{TUSLA} the use of stochastic gradients provides another layer of difficulty which could be a path for future work.
\section{Moment bounds}
In this section key properties regarding the growth and dissipativity of the drift coefficient of our tamed scheme are presented, enabling the derivation of uniform in the number of iterations moment bounds. It is important to point out, as it will be used in the following bounds, that due to \ref{ass-initial cond}, there holds
\[\E|\theta_0|^p<\infty \quad \forall p>0.\]
\begin{lemma}\label{growth-tamed}
  Let Assumption \aref{ass-derivbound}  hold. Then,  for $\lambda<1$, there holds 
  \[|h_\lambda(x)|\leq a|x|+\frac{1}{\sqrt{\lambda}}C_h.\]
  where $C_h:=L +a$ .
  \begin{proof}
      Writing \[|h_\lambda(x)|\leq a|x|+ \frac{a|x|+|h(x)|}{1+\sqrt{\lambda}|x|^{2l}}\leq a|x|+\frac{a}{\sqrt{\lambda}} \frac{|x|}{1+|x|^{2l}} +\frac{1}{\sqrt{\lambda}}\frac{|h(x)|}{1+|x|^{2l}} .\]
      Using the bound in \aref{ass-derivbound} the result follows easily.
  \end{proof}
  \end{lemma}
\begin{lemma}\label{dissipativity-tamed}
   Let Assumption \aref{ass-2dissip} hold. Then, $\langle h_\lambda(x),x \rangle \geq \frac{a}{2}|x|^2-b$.
\end{lemma}
\begin{proof}
    Recall that $h(x)=ax +f(x) $
    Then, \[\langle h_\lambda(x),x\rangle=a|x|^2 +\langle f_\lambda(x),x\rangle \]
    Then, if $\langle f(x),x\rangle\geq 0$ it easily follows that $\langle f_\lambda(x),x\rangle\geq 0$ so
    \begin{equation}\label{eq-dissipt1}
        \langle h_\lambda(x),x \rangle \geq \frac{a}{2}|x|^2-b.
    \end{equation}
    On the other hand if $\langle f(x),x\rangle<0 $ one obtains that
    \[\langle f_\lambda(x),x\rangle \geq \langle f(x),x\rangle \] so
    \begin{equation}\label{eq-dissipt2}
        \langle h_\lambda(x),x\rangle \geq \langle h(x),x\rangle \geq \frac{a}{2}|x|^2-b.
    \end{equation}
\end{proof}
\begin{lemma}\label{2moments}
Let Assumptions \aref{ass-derivbound}-\aref{ass-2dissip} hold and $\lambda<\min\{1,\frac{1}{4a}\}.$ Then,
\[\sup_n \E |\tn|^2<\bar{C}_2\] where \begin{equation}
    \bar{C}_2=\E |\theta_0|^2 +2\left(2C_h^2+\frac{2}{\beta}d+2b\right)/a.
\end{equation}
\end{lemma}
\begin{proof}
\[\begin{aligned}
\E [|\theta^\lambda_{n+1}|^2|\tn]&=\E \left[\left|\tn-\lambda h_\lambda(\tn)+\sqrt{\frac{2\lambda}{\beta}}\xi_{n+1}\right|^2\big |\tn\right]\\&=|\tn-\lambda h_\lambda(\tn)|^2 +\frac{2\lambda}{\beta}\E |\xi_{n+1}|^2
\\&=|\tn|^2-2\lambda\langle \tn,h_\lambda(\tn)\rangle +\lambda^2 |h_\lambda(\tn)|^2+\frac{2\lambda}{\beta}d
\end{aligned}\]
where the second step was derived from the independence of $\tn$ and $\xi_{n+1}.$
Using the dissipativity and the growth properties of $h_\lambda$ on obtains
\[\E [|\bar{\theta}_{n+1}^\lambda|^2\big|\tn]\leq |\tn|^2 -\lambda a|\tn|^2 +2\lambda B +2\lambda^2 a^2 |\tn|^2+ 2\lambda C_h^2 +\lambda \frac{2}{\beta}d.\]
Since $2\lambda <\frac{1}{2a}$ there holds
\[\E [|\bar{\theta}_{n+1}^\lambda|^2\big|\tn]\leq (1-\lambda \frac{a}{2})|\tn|^2 +\lambda \left(2C_h^2+\frac{2}{\beta}d+2b\right).\]
Taking expectations yields
\[\E |\bar{\theta}_{n+1}^\lambda|^2\leq (1-\lambda \frac{a}{2})\E |\tn|^2 +\lambda \left(2C_h^2+\frac{2}{\beta}d+2b\right).\]
By induction,
\[\E |\bar{\theta}_{n+1}^\lambda|^2\leq (1-\lambda \frac{a}{2})^n \E |\theta_0|^2 +2\left(2C_h^2+\frac{2}{\beta}d+2b\right)/a.\]
\end{proof}

\begin{lemma}\label{highmom}
Let Assumptions \aref{ass-derivbound}-\aref{ass-initial cond} hold.
For $\lambda<\min\{1,\frac{1}{4a}\}.$
\[\sup_n \E |\tn|^{2p}\leq \bar{C}_p.\]
where $\bar{C}_p:=\E |\theta_0|^{2p} +\frac{4}{A}(p(2p-1)2^{2p-2}\beta^{-1}d N_{p,\beta,d}^{2p-1}+C_{p,d,\beta})$ where the rest of the constants are given explicitly in the proof.
\end{lemma}
\begin{proof}
Let $\Delta_n=\tn-\lambda h_\lambda(\tn)$ and $\Xi^\lambda_n=\sqrt{\frac{2\lambda}{\beta}}\xi_{n+1}$.
Using the inequality \[|x+y|^{2 p} \leq|x|^{2 p}+2 p|x|^{2 p-2}\langle x, y\rangle+\sum_{k=2}^{2 p}\left(\begin{array}{c}
2 p \\
k
\end{array}\right)|x|^{2 p-k}|y|^k,\] one deduces
\begin{equation}
    \begin{aligned}\label{eq-mombase}
\begin{aligned}
\hspace{-28pt}\E \left[|\bar{\theta}_{n+1}|^{2p}\big| \tn\right]
&\leq\left|\Delta^\lambda_n\right|^{2 p}+2 p\left|\Delta^\lambda_n\right|^{2 p-2} \mathbb{E}\left[\left\langle\Delta^\lambda_n, \Xi^\lambda_n\right\rangle \mid \bar{\theta}_n^\lambda\right]\\&+\sum_{k=2}^{2 p}\left(\begin{array}{c}
2 p \\
k
\end{array}\right) \mathbb{E}\left[\left|\Delta^\lambda_n\right|^{2 p-k}\left|\Xi^\lambda_n\right|^k \mid \bar{\theta}_n^\lambda\right]. \\
&=\left|\Delta^\lambda_n\right|^{2 p}+\sum_{k=0}^{2 p-2}\left(\begin{array}{c}
2 p \\
k+2
\end{array}\right) \mathbb{E}\left[\left|\Delta^\lambda_n\right|^{2 p-2-k}\left|\Xi^\lambda_n\right|^k\left|\Xi^\lambda_n\right|^2 \mid \bar{\theta}_n^\lambda\right] \\
&=\left|\Delta^\lambda_n\right|^{2 p}+\sum_{k=0}^{2 p-2} \frac{2 p(2 p-1)}{(k+2)(k+1)}\left(\begin{array}{c}
2 p-2 \\
k
\end{array}\right) \mathbb{E}\left[\left|\Delta^\lambda_n\right|^{2 p-2-k}\left|\Xi^\lambda_n\right|^k\left|\Xi^\lambda_n\right|^2 \mid \bar{\theta}_n^\lambda\right] \\
&\leq\left|\Delta^\lambda_n\right|^{2 p}+p(2 p-1) \sum_{k=0}^{2 p-2}\left(\begin{array}{c}
2 p-2 \\
k
\end{array}\right) \mathbb{E}\left[\left|\Delta^\lambda_n\right|^{2 p-2-k}\left|\Xi^\lambda_n\right|^k\left|\Xi^\lambda_n\right|^2 \mid \bar{\theta}_n^\lambda\right] \\
&=\left|\Delta^\lambda_n\right|^{2 p}+p(2 p-1) \mathbb{E}\left[\left(\left|\Delta^\lambda_n\right|+\left|\Xi^\lambda_n\right|\right)^{2 p-2}\left|\Xi^\lambda_n\right|^2 \mid \bar{\theta}_n^\lambda\right] \\
&=\left|\Delta^\lambda_n\right|^{2 p}+p(2 p-1) 2^{2 p-3} \mathbb{E}\left[\left(\left|\Delta^\lambda_n\right|^{2 p-2}+\left|\Xi^\lambda_n\right|^{2 p-2}\right)\left|\Xi^\lambda_n\right|^2 \mid \bar{\theta}_n^\lambda\right] \\
&=\left|\Delta^\lambda_n\right|^{2 p}+p(2 p-1) 2^{2 p-3}\left|\Delta^\lambda_n\right|^{2 p-2} \mathbb{E}\left[\left|\Xi^\lambda_n\right|^2\right]+p(2 p-1) 2^{2 p-3} \mathbb{E}\left[\left|\Xi^\lambda_n\right|^{2 p}\right] \\
&=\left|\Delta^\lambda_n\right|^{2 p}+p(2 p-1) 2^{2 p-2} \lambda \beta^{-1} d\left|\Delta^\lambda_n\right|^{2 p-2}\\&+ p(2 p-1) 2^{4 p-3}\left(\lambda \beta^{-1}\right)^p p !\left(\begin{array}{c}
\frac{d}{2}+p-1 \\
p
\end{array}\right)
\end{aligned}
\end{aligned}
\end{equation}
Using the arguments of the proof of Lemma \ref{2moments} one deduces that \begin{equation}
    |\Delta^\lambda_n|^2\leq (1-\lambda \frac{a}{2})|\tn|^2 + \lambda b'
\end{equation}
where $b'=2(C_h^2+B).$
Writing \[\begin{aligned}
 |\Delta_n|^{2p}&= \left[(1-\lambda \frac{a}{2})|\tn|^2 +\lambda b' \right]^{p}
 \\&\leq \sum_{k=0}^p (1-\lambda \frac{a}{2})^k |\tn|^{2k}  \lambda^{p-k
}(b')^{p-k}
\\&\leq (1-\lambda \frac{a}{2})|\tn|^{2p} + \sum_{k=0}^{p-1} (1-\lambda \frac{a}{2})^k |\tn|^{2k}  \lambda^{p-k
}(b')^{p-k}
\\&\leq  (1-\lambda \frac{a}{4})|\tn|^{2p} - \frac{\lambda a}{4}|\tn|^{2p} + \lambda \sum_{k=0}^{p-1}|\tn|^{2k} b'^{p-k}
\\&\leq(1-\lambda \frac{a}{4})|\tn|^{2p}+ \lambda  \sum_{k=0}^{p-1} \left[|\tn|^{2k} b'^{p-k}-\frac{a}{4p}|\tn|^{2p}\right]
\end{aligned}\]
Let $M_{p}:=\max{1,\max_{0\leq k\leq p-1} \left(\frac{2p}{A}b'^{p-k} \right)^{\frac{1}{2(p-k)}}}$
If $|\tn|\geq M_{p}$ the second term is negative
so
\[|\Delta_n^\lambda|^{2p}\mathds{1}_{|\tn|\geq M_{p}}\leq (1-\frac{\lambda a}{4})|\tn|^{2p}\]
and \[|\Delta_n^\lambda|^{2p}\mathds{1}_{|\tn|<M_{p}}\leq (1-\frac{\lambda a}{4})|\tn|^{2p}+\lambda p M_{p}^{2p-1}(1+b')^p\] which leads to
\begin{equation}\label{eq-dn}
    |\Delta_n^\lambda|^{2p}\leq (1-\frac{\lambda a}{4})|\tn|^{2p} +\lambda c(p)
\end{equation}
where $c(p)= p M_{p}^{2p-1}(1+b')^p$.
Applying \eqref{eq-dn} for $p-1$ one obtains
\begin{equation}\label{eq-dn1}
    |\Delta_n^\lambda|^{2p-2}\leq |\tn|^{2p-2} +\lambda c(p-1).
\end{equation}
Inserting \eqref{eq-dn}, \eqref{eq-dn1} into \eqref{eq-mombase} yields
\[\begin{aligned}
 \E\left [|\bar{\theta}^\lambda_{n+1}|^{2p}\big| \tn \right]&\leq (1-\frac{\lambda a}{4})|\tn|^{2p}  +\lambda p(2p-1)2^{2p-2} \beta^{-1}d  |\tn|^{2p-2} \\&+\lambda \left(c(p)+ p(2p-1)2^{2p-2} \beta^{-1}d c(p-1)\right)
 \\&+  p(2 p-1) 2^{4 p-3}\left(\lambda \beta^{-1}\right)^p p !\left(\begin{array}{c}
\frac{d}{2}+p-1 \\
p
\end{array}\right)
\\&=(1-\frac{\lambda a}{8})|\tn|^{2p} +\left(-\frac{\lambda a}{8}|\tn|^{2p} + \lambda p(2p-1)2^{2p-2} \beta^{-1}d  |\tn|^{2p-1}\right)\\&+ \lambda C_{p,d,\beta}.
\end{aligned}\]
Let $N_{p,\beta,d}:= \frac{8}{a}p(2p-1)2^{2p-2} \beta^{-1}d .$
If $|\tn|^2\geq N_{p,\beta,d}$ then the middle term is negative otherwise, \\it is bounded by above by $\lambda p(2p-1)2^{2p-2}\beta^{-1}d N_{p,\beta,d}^{p-1}$ so
 \begin{equation}\label{eq-momsemifinal}
   \E\left [|\theta^\lambda_{n+1}|^{2p}\big| \tn \right]\leq (1-\lambda \frac{a}{8})|\tn|^{2p}   + \lambda (p(2p-1)2^{2p-2}\beta^{-1}d N_{p,\beta,d}^{p-1}+C_{p,d,\beta}).
 \end{equation}
 Taking expectations and iterating in $n$ in \eqref{eq-momsemifinal} yields
 \[\E|\theta_n^\lambda|^{2p}\leq (1-\lambda \frac{a}{8})\E |\theta_0|^{2p} +  \frac{8}{a}(p(2p-1)2^{2p-2}\beta^{-1}d N_{p,\beta,d}^{p-1}+C_{p,d,\beta}).\]
\end{proof}
\section{Establishing a key differential inequality regarding KL- divergence}
The goal of this Section is to establish a differential inequality that will be the basis for our analysis.
We define the continuous-time interpolation of our algorithm given as 
\begin{equation}
    \theta_t=\theta_{k\lambda}-(t-k\lambda) h_\lambda(\theta_{k\lambda}) + \frac{\sqrt{2}}{\beta}(B_{t}-B_{k\lambda}), \quad \forall t\in [k\lambda,(k+1)\lambda]
\end{equation}
and $\theta_0=\bar{\theta_0}.$

That way \[\mathcal{L}(\theta_{k\lambda})=\mathcal{L}(\bar{\theta}^\lambda_{k}) \quad \forall k \in \mathbb{N}.\]
We define the marginal distribution of $\theta_t$ as $\pt$.
     One notices that, since conditioned on $\theta_{k\lambda}$, $\theta_t$ is a Gaussian its conditional distribution  is given by \[\ptk(x|y)= C e^{-\frac{\sqrt{t-k\lambda}}{2}|x-\mu(t,y)|^2}\]
     where $\mu(t,y)= y-(t-k\lambda)h_\lambda(y)$ and $C$ some normalizing constant. One further notes that, as $\ptk(x|y)$ can be viewed as a distribution of a process satisfying a Langevin SDE with constant drift $-h_\lambda(y)$ and initial condition $y$, i.e
     \[\begin{aligned}
         d\hat{\mu}_t&=-h_\lambda(y)dt + \sqrt{\frac{2}{\beta}}dB_t, \quad \forall t \in (k\lambda,(k+1)\lambda]
         \\ \hat{\mu}_{k\lambda}&=y
     \end{aligned}\] it satisfies the following Fokker-Planck PDE:
    \begin{equation}\label{eq-FP}
       \frac{\partial \ptk(x|y)}{\partial t}=div\left(\ptk(x|y) h_\lambda(y)\right) +\Delta_x \ptk(x|y).
    \end{equation}
    \newtheorem{1exch}[Def1]{Lemma}
    \begin{1exch}\label{1exch}
    Let Assumptions \aref{ass-derivbound}-\aref{ass-initial cond} hold.
Then,   \[\E \left(\frac{\partial \ptk(x|\theta_{k\lambda})}{\partial t}\right)=\frac{\partial\pt}{\partial t}(x).\]
    \end{1exch}
\begin{proof}
Analysing the left hand side of the equation one deduces the following:

In a neighbourhood of $t$, for fixed $x$,  $\frac{\partial \ptk(x|y)}{\partial t}$ decays exponentially with $y$ and since $\pkl(y)\leq Ce^{-r|y|^2}$ due to Lemma \ref{pt-decay} one can exchange the derivative with the integral in the following expression
\[\frac{\partial}{\partial t} \int_{\mathbb{R}^d} \pkl(y) \ptk(x|y) dy=\int_{\mathbb{R}^d} \pkl(y) \frac{\partial \ptk(x|y)}{\partial t}dy.\]
Noticing that \[\frac{ \partial \pt}{\partial t} (x)=\frac{\partial}{\partial t} \int_{\mathbb{R}^d} \pkl(y) \ptk(x|y) dy\] and
\[\int_{\mathbb{R}^d} \pkl(y)\frac{\partial \ptk(x|y)}{\partial t}dy=\E \left(\frac{\partial \ptk(x|\theta_{k\lambda})}{\partial t}\right)\] yields the result.
\end{proof}
\newtheorem{2exch}[Def1]{Lemma}
\begin{2exch}\label{2exch}
Let Assumptions \aref{ass-derivbound}-\aref{ass-initial cond} hold. Then,\[\E \left( div_x \left (\ptk(x|\theta_{k\lambda}) h_\lambda(\theta_{k\lambda})\right)\right)=div_x \left(\pt(x) \E \left (h_\lambda(\theta_{k\lambda})\big| \theta_t=x\right)\right). \]
\end{2exch}
\begin{proof}
Since $\pt$ decays exponentially with $y$ and for fixed $t$, in a neighbourhood of $x$ $\nabla\ptk(x|y)$ is at most linear in $y$ and $h_\lambda$ has at most linear growth, this enables the interchange of integral and derivative with respect to $x$ in the following expression
\[\int_{\mathbb{R}^d} \pkl(y) div _x\left(\ptk(x|y)h_\lambda(y)\right)dy=div_x \int_{\mathbb{R}^d}\pkl(y) \ptk(x|y) h_\lambda(y) dy\]
Since \[\E \left( div \left (\ptk(x|\theta_{k\lambda}) h_\lambda(\theta_{k\lambda})\right)\right)=\int_{\mathbb{R}^d} \pkl(y) div_x \left(\ptk(x|y)h_\lambda(y)\right)dy\] and due to Bayes theorem
\[\begin{aligned}
div_x \int_{\mathbb{R}^d} \pkl(y) \ptk(x|y)h_\lambda (y) dy&=div _x\int_{\mathbb{R}^d} \pt(x)
\hat{\pi}_{\theta_{k\lambda}|\theta_t}(y|x) h_\lambda(y)dy\\&=div _x\left(\pt(x) \E \left (h_\lambda(\theta_{k\lambda})\big| \theta_t=x\right)\right)
\end{aligned}\]
and the result immediately follows.
\end{proof}
\newtheorem{3exch}[Def1]{Lemma}
\begin{3exch}\label{3exch}
Let Assumptions \aref{ass-derivbound}-\aref{ass-initial cond} hold.
Then, \[\E \left(\Delta_x \ptk (x|\theta_{k\lambda} )\right)=\Delta \pt(x).\]
\end{3exch}
\begin{proof}
Noting that by definition \[\E \left(\Delta_x \ptk (x|\theta_{k\lambda} )\right)=\int_{\mathbb{R}^d} \Delta_x(\ptk(x|y))\pkl(y)dy\]
and
\[\Delta_x \pt(x)=\Delta_x \int_{\mathbb{R}^d} \ptk(x|y)\pkl(y)dy\]
it suffices to prove that \[ \int_{\mathbb{R}^d} \Delta_x(\ptk(x|y))\pkl(y)dy=\Delta_x \int_{\mathbb{R}^d} \ptk(x|y)\pkl(y)dy.\]
By simple computations for the Gaussian distribution one deduces that $|\nabla_x \log \ptk(x|y)|$, $\Delta_x \log \ptk(x|y)$ have at most linear growth with respect to $y$ in a neighbourhood of $x$ .
Writing \[\Delta_x \ptk(x|y)=\left(\Delta_x \log \ptk(x|y)+|\nabla_x \log \ptk(x|y)|^2\right)\ptk(x|y)\] one deduces that in a neighbourhood of $x$,
the integrand in the first term is dominated by a function of the form $C(1+|y|^2)e^{-c|y|^2}.$ Applying the dominated convergence theorem enables the exchange of the integral and the Laplacian which completes the proof.
\end{proof}
\newtheorem{exchres}[Def1]{Corollary}
\begin{exchres}\label{exchres}
Let Assumptions \aref{ass-derivbound}-\aref{ass-initial cond} hold.
Then, \[\frac{\partial \pt}{\partial t}(x)=div_x\left(\pt(x) \E \left (h_\lambda(\theta_{k\lambda})\big| \theta_t=x\right)\right)+ \frac{1}{\beta} \Delta\pt(x) \quad \forall t \in [k\lambda,(k+1)\lambda] \]
\end{exchres}
\begin{proof}
Taking expectations in \eqref{eq-FP} and combining Lemmas \ref{1exch},\ref{2exch} and \ref{3exch} yields the result.
\end{proof}
\newtheorem{timechange1}[Def1]{Lemma}
\begin{timechange1}\label{timechange1}
Let Assumptions \aref{ass-derivbound}-\aref{ass-initial cond} hold. Then, there exist $C,k,r'>0$ indepent of $x$, uniform in a small neighbourghood of $t$ such that 
\[div_x \left(\pt(x) \E \left (h_\lambda(\theta_{k\lambda})\big| \theta_t=x\right)\right)+\frac{1}{\beta}\Delta\pt\leq C(1+|x|^k)e^{-r'|x|^2}\]
\end{timechange1}
\begin{proof}
Writing, due to Bayes' theorem, 
\[\begin{aligned}
&div_x \left(\pt(x) \E \left (h_\lambda(\theta_{k\lambda})\big| \theta_t=x\right)\right)=\int_{\mathbb{R}^d} \pkl(y) div _x\left(\ptk(x|y)h_\lambda(y)\right)dy\\&\leq Ce^{-c|x|^2+|x|}\int_{\mathbb{R}^d}e^{-r|y|^2}|y| dy
\end{aligned} \] for some $C,c,r>0$ where the last step is a result of the Gaussian expression of the conditional density, the linear growth of $h_\lambda$ and the exponential decay of $\pkl$ given in Lemma \ref{pt-decay}.

For the second term, writing \[\Delta\pt = \pt \left( |\nabla \log \pt|^2+\Delta \log \pt\right)\]  the result follows due to  Lemmas \ref{pt-decay}, \ref{gradlog1 growth}, \ref{gradlog2 growth}.
\end{proof}
\newtheorem{cor-timechange}[Def1]{Corollary}
\begin{cor-timechange}\label{cor-timechange}
Let Assumptions \aref{ass-derivbound}-\aref{ass-initial cond} hold. Then,
\[\frac{d}{dt}H_{\pi_\beta}(\pt)=\int_{\mathbb{R}^d} \frac{\partial \pt(x)}{\partial t}(1+\log\pt(x)-\log \pi_\beta)dx\]
\end{cor-timechange}
\begin{proof}
Noting that $\log\pt,\log\pi$ have polynomial growth, due to Lemma \ref{timechange1},\\ $\frac{\partial \pt(x)}{\partial t}(1+\log\pt(x)-\log \pi_\beta)$ can be dominated by an $L^1$ integrable function over small neighbourhood of $t$, thus using the dominated convergence theorem one deduces the exchange of derivative and integration i.e
\[\begin{aligned}
 \int_{\mathbb{R}^d} \frac{\partial \pt(x)}{\partial t}(1+\log\pt(x)-\log \pi_\beta(x))dx&=\int_{\mathbb{R}^d}\frac{\partial}{\partial t}\left (\pt(x)\log \frac{\pt(x)}{\pi_\beta(x)}\right)dx\\&=\frac{d}{dt}\int_{\mathbb{R}^d} \pt(x)\log \frac{\pt(x)}{\pi_\beta(x)}dx\\&=\frac{d}{dt}H_{\pi_\beta}(\pt).
\end{aligned}\]
\end{proof}
\newtheorem{divergence}[Def1]{Corollary}
\begin{divergence}\label{divergence}
Let Assumptions \aref{ass-derivbound}-\aref{ass-initial cond} hold.
Let $k\in \mathbb{N}.$
Then, for every $t \in [k\lambda,(k+1)\lambda],$
\[\begin{aligned}
 \frac{d}{dt}  H_{\pi_\beta}(\pt)&=-\int_{\mathbb{R}^d}  \langle \pt(x)E \left (h_\lambda(\theta_{k\lambda})\big| \theta_t=x\right)+\frac{1}{\beta}\nabla \pt(x),\nabla \log\pt(x)-\nabla \log \pi_\beta \rangle dx .
\end{aligned}\]
\end{divergence}
\begin{proof}
Recall that from Lemma \ref{cor-timechange}, there holds
\begin{equation}
     \frac{d}{dt}  H_{\pi_\beta}(\pt)=\int_{\mathbb{R}^d}\frac{\partial \pt(x)}{\partial t}(1+\log\pt(x)-\log \pi_\beta)dx.
\end{equation}
Let \[F_t(x)=\pt(x)E \left (h_\lambda(\theta_{k\lambda})\big| \theta_t=x\right)+\frac{1}{\beta}\nabla \pt(x)\] and \[g_t(x)=1+\log\pt-\log\pi.\]
Recall from Corollary \ref{exchres} that \[\frac{\partial \pt}{dt}(x) =div_x (F_t) (x)\]
Since $\nabla \pt=\pt \nabla \log \pt$ using the   Lemmas \ref{pt-decay}, \ref{gradlog1 growth}, \ref{gradlog2 growth} ,\ref{timechange1} one deduces that there exists constants $C$, $q$ ,$r$>0 independent of $x$, uniform in a small neighbourhood of $t$, such that
\begin{equation}\label{eq-divbound}
    \max\{|F_t(x)g_t(x)|,|div(F_t)(x)g_t(x)|,|\langle F_t(x)\nabla g_t(x)\rangle|\}\leq C(1+|x|^q)e^{-r|x|^2}.
\end{equation}
We drop the dependence of the constants on $t$ since we want to integrate with respect to $x.$
Let $R>0$ and $v(x)$ the normal unit vector on $\partial B(0,R)$. Due to \eqref{eq-divbound}
\begin{equation}\label{eq-limit}
    \int_{\partial B(0,R)} \langle g_t(x)F_t(x),v(x)\rangle dx \leq R^d C(1+|R|^q)e^{-r|R|^2}.
\end{equation}
Since $div(F_t) g_t$, $\langle F_t,\nabla_x g_t \rangle$ are integrable (in view of \eqref{eq-divbound}) applying the divergence theorem on $B(0,R)$ there holds
\begin{equation}
    \int_{B(0,R)} div_x(F_t)(x) g_t(x) dx= \int_{\partial B(0,R)} \langle g_t(x)F_t(x),v(x)\rangle dx- \int_{B(0,R)}\langle F_t(x)\nabla_x g_t(x)\rangle dx.
\end{equation}
As a result,
\[\begin{aligned}
\hspace{-20pt} \int_{\mathbb{R}^d} div(F_t)(x) g_t(x)dx&=\lim_{R\rightarrow \infty} \int_{B(0,R)} div_x(F_t)(x) g_t(x) dx\\&=\lim_{R\rightarrow \infty}\left( \int_{\partial B(0,R)} \langle g_t(x)F_t(x),v(x)\rangle dx- \int_{B(0,R)}\langle F_t(x)\nabla_x g_t(x)\rangle\right)dx\\&=0-\lim_{R\rightarrow \infty}\int_{B(0,R)}\langle F_t(x)\nabla_x g_t(x)\rangle dx=-\int_{\mathbb{R}^d}\langle F_t(x)\nabla_x g_t(x)\rangle dx.
\end{aligned}\]
\end{proof}
\newtheorem{Interpolation ineq}[Def1]{Theorem}
\begin{Interpolation ineq}\label{Interpolation ineq}
Let Assumptions \aref{ass-derivbound}-\aref{ass-initial cond} and \lref{ass-LSI}  hold.
Then, for $\lambda<\lambda_{\max}$ and for every $t\in[k\lambda,(k+1)\lambda]$ ,$k\in \mathbb{N},$ there holds
\[\frac{d}{dt}H_{\pi_\beta}(\pt)\leq-\frac{3}{4} I_{\pi_\beta} (\pt) + \beta  \E | h(\theta_t)-h_{\lambda}(\theta_{k\lambda})|^2\]
where $I_{\pi_\beta}$ is given in Definition \ref{LSIdef}.
\end{Interpolation ineq}
\begin{proof}
Using Corollary \ref{divergence}, for all $t \in [k\lambda,(k+1)\lambda],$
    \[\begin{aligned}
\hspace{-20pt} \frac{d}{dt}H_{\pi_\beta}(\pt)&=-\int_{\mathbb{R}^d}  \langle \pt(x)E \left (h_\lambda(\theta_{k\lambda})\big| \theta_t =x \right)+\frac{1}{\beta}\nabla \pt(x),\nabla \log\pt(x)-\nabla \log \pi_\beta(x) \rangle dx
 \\&=-\int_{\mathbb{R}^d} \pt(x) \langle E \left (h_\lambda(\theta_{k\lambda})\big| \theta_t=x\right)+\frac{1}{\beta}\nabla \log \pt(x),\nabla \log\pt(x)-\nabla \log \pi_\beta(x) \rangle dx
 \\&=-\int_{\mathbb{R}^d} \pt(x) \langle  E \left (h_\lambda(\theta_{k\lambda})\big| \theta_t=x\right)+\frac{1}{\beta}\nabla \log \pi_\beta,\nabla \log\pt(x)-\nabla \log \pi_\beta(x) \rangle dx
 \\&-\frac{1}{\beta}\int_{\mathbb{R}^d} \pt |\nabla \log\pt(x)-\nabla \log \pi_\beta(x) |^2 dx
 \\&=-I_{\pi_\beta}(\pt)-\int_{\mathbb{R}^d} \pt(x) \langle  E \left (h_\lambda(\theta_{k\lambda})-h(x)\big| \theta_t=x\right),\nabla \log\pt(x)-\nabla \log \pi_\beta(x) \rangle dx
 \\&=-I_{\pi_\beta}(\pt)-\int_{\mathbb{R}^d} \pt(x) \langle  E \left (h_\lambda(\theta_{k\lambda})-h(\theta_t))\big| \theta_t=x\right),\nabla \log\pt(x)-\nabla \log \pi_\beta(x) \rangle dx
 \\&\leq - I_{\pi_\beta}(\pt)+\beta\int_{\mathbb{R}^d} \pt(x) \left | E \left (h_\lambda(\theta_{k\lambda})-h(\theta_t))\big| \theta_t=x\right)\right|^2dx +\frac{1}{4}I_{\pi_\beta}(\pt)
 \\&=-\frac{3}{4}I_{\pi_\beta}(\pt) + \beta \int_{\mathbb{R}^d} \pt(x)\left |\int_{\mathbb{R}^d} 
\hat{\pi}_{\theta_{k\lambda}|\theta_t}(y|x) (h_\lambda(y)-h(x))dy \right |^2 dx\\&\leq 
-\frac{3}{4}I_{\pi_\beta}(\pt)+ \beta \int_{\mathbb{R}^d} \pt(x) \int_{\mathbb{R}^d} 
\hat{\pi}_{\theta_{k\lambda}|\theta_t}(y|x) \left|h_\lambda(y)-h(x)\right|^2dy dx
 \\&=-\frac{3}{4}I_{\pi_\beta}(\pt)+\beta\E | h_\lambda(\theta_{k\lambda})-h(\theta_t)|^2
\end{aligned}\]
where the first inequality was obtained using Young inequality and the second using Jensen's.
\end{proof}
\section{Proof of non-asymptotic bounds for sampling: Theorem \ref{H-rate}}
\newtheorem{tamed to h}[Def1]{Lemma}
\begin{lemma}\label{tamingerror}
Let Assumptions \aref{ass-derivbound}-\aref{ass-initial cond} hold.
    Then,
    \[\E |h_\lambda(\theta_{k\lambda})-h(\theta_{k\lambda})|^2\leq C_{tam} \lambda \quad \forall k\in \mathbb{N}, \]
    where $C_{tam}:= 16\left( L^2 (\bar{C}_{4l}+1)+ \bar{C}_{2l+1}\right)$  and the constants are given in Lemma \ref{highmom}.
\end{lemma}
\begin{proof}
    For every $x\in \mathbb{R}^d$,
    \[|h_\lambda(x)-h(x)|=\left|(h(x)-ax)(1-\frac{1}{1+\sqrt{\lambda}|x|^{2l}})\right|^2\leq \lambda \left|(|h(x)|+|x|)|x|^{2l}\right|^2\]
    so
    \[\E \left|h(\theta_{k\lambda})-h_\lambda(\theta_{k\lambda})\right|^2\leq \lambda \E \left|\left(|h(\Bar{\theta}_k)|+|\Bar{\theta}_k|\right)|\Bar{\theta}_k|\right|^2\leq  16\left( L^2 (\bar{C}_{4l}+1)+ \bar{C}_{2l+1}\right) \lambda.\]
    where the constants are given in Lemma \ref{highmom}.
\end{proof}
\newtheorem{onesteperror}[Def1]{Lemma}
\begin{lemma}\label{onesteperror}
Let Assumptions \aref{ass-derivbound}-\aref{ass-initial cond} hold.  Let $k\in \mathbb{N}.$ Then,
\[\E |h(\theta_{k\lambda})-h(\theta_t)|^2\leq C_{onestep}\lambda, \quad \forall t\in[k\lambda,(k+1)\lambda], \]
where $C_{onestep}$ is given in the proof is independent of $\lambda$ and depends polynomially on the dimension.
\end{lemma}
\begin{proof}
Let $t \in[ k\lambda,(k+1)\lambda].$
    First of all, one needs to bound the one step error $\E |\theta_t-\theta_{k\lambda}|^{2p}$ for different values of $p\in \mathbb{N}$.
    \[\begin{aligned}
        \E |\theta_t-\theta_{k\lambda}|^{2p}&\leq 2^{2p} \lambda^{2p}\E |h_\lambda(\theta_{k\lambda})|^{2p} +2^{p} \lambda^p \E |Z|^{2p}
        \\&\leq 2^{p} \lambda^{p}\E \left(2a|\theta_{k\lambda}|+|h(\theta_{k\lambda}|\right)|^{2p}+ 2^p\lambda^{p}\E |Z|^{2p}
        \\&\leq \lambda^{p} 2^p \lambda^p (1+2a)^{2p} \left(\E \left(1+|\theta_{k\lambda}|^{2l+1}\right)^{2p} +\E |Z|^{2p})\right)
        \\&\leq \lambda^p C_{1,p}
    \end{aligned}\]
    where $C_{1,p}=\mathcal{O}\left(d^{p(2l+1)}\right),$ which is derived by the moment bounds of the Gaussian, the fact that $\mathcal{L}(\theta_{k\lambda})=\mathcal{L}(\bar {\theta}^\lambda_{k})$ and the result in Lemma \ref{highmom}.
    In addition,
    \[\E (1+|\theta_t|+|\theta_{k\lambda}|)^{2p}\leq 3^{2p}\left(1+2^{2p}\E|\theta_{k\lambda}|^{2p}+\E |\theta_t-\theta_{k\lambda}|^{2p}\right)\leq C_{2,p},\]
    where $C_{2,p}\leq \mathcal{O}\left(d^{p(2l+1)}\right)$.
    Using the local Lipschitz property of $h$ and Cauchy -Swartz inequality one obtains
    \[\begin{aligned}
        \E |h(\theta_{k\lambda})-h(\theta_t)|^2&\leq
        L'^2 \E (1+|\theta_{k\lambda}|+|\theta_t|)^{2l'} |\theta_t-\theta_{k\lambda}|^2
        \\&\leq 2^{2l'}L'^2 \sqrt{\E (1+|\theta_{k\lambda}|+|\theta_t-\theta_{k\lambda}|)^{4l'}}\sqrt{\E|\theta_t-\theta_{k\lambda}|^4 }
        \\&\leq 2^{2l'} L'^2 \sqrt{C_{2,2l'}}\sqrt{C_{1,2}}\lambda.
    \end{aligned}\]
\end{proof}
\begin{proof}[Proof of Theorem \ref{H-rate}]
Setting $\dot{c}=\frac{3}{2}C_{LSI}$ and using the differential inequality obtained in Theorem \ref{Interpolation ineq} one obtains
\[\begin{aligned}
 \frac{d}{dt}H_{\pi_\beta}(\pt) &\leq -\frac{3}{4} I_{\pi_\beta}(\pt) +   \beta \E |\hl(\theta_{k\lambda}-h(\theta_t)|^2
\\&\leq -\Dot{c} H_{\pi_\beta} (\pt) + 2\beta \E |\hl(\theta_{k\lambda})-h(\theta_{k\lambda})|^2 + 2 \beta \E | h(\theta_{k\lambda})-h(\theta_t)|^2
\\&\leq -\Dot{c} H_{\pi_\beta} (\pt) + \beta \hat{C} \lambda
\end{aligned}\]
where $\hat{C}=2C_{onestep}+2C_{tam}$
where the first term has been bounded using the Log-Sobolev inequality and the rest of the terms using the one-step error in Lemma \ref{onesteperror} and the taming error in Lemma \ref{tamingerror}.
Splitting the terms  one obtains
\[\begin{aligned}
 \left( \frac{d}{dt}H_{\pi_\beta}(\pt)+\Dot{c} H_{\pi_\beta} (\pt)\right) e^{\Dot{c}t} \leq e^{\Dot{c}t} \beta \hat{C} \lambda
\end{aligned}\]
Integrating over $[k\lambda,t]$ yields
\[\begin{aligned}
 e^{\Dot{c}t}H_{\pi_\beta}(\pt)- e^{\Dot{c}k\lambda} H_{\pi_\beta}(\hat{\pi}_{k\lambda})\leq \frac{\beta \hat{C}}{\Dot{c}}\lambda (e^{\Dot{c} t}-e^{\Dot{c}k\lambda})
\end{aligned}\]
which implies
\begin{equation}
    H_{\pi_\beta}(\pt) \leq e^{\Dot{c}(k\lambda-t)} H_{\pi_\beta}(\hat{\pi}_{k\lambda}) +\frac{\beta \hat{C}}{\Dot{c}}\lambda (1-e^{\Dot{c}(k\lambda-t)}).
\end{equation}
Setting $t=n\lambda$ and $k=(n-1)$ leads to
\[H_{\pi_\beta}(\hat{\pi}_{n\lambda})\leq e^{-\Dot{c}\lambda } H_{\pi_\beta}(\hat{\pi}_{(n-1)\lambda}) +  \frac{\beta \hat{C}}{\Dot{c}}\lambda (1-e^{-\Dot{c}\lambda})\]
so by iterating over $n$,
\[H_{\pi_\beta} (\hat{\pi}_{n\lambda})\leq e^{-\dot{c} \lambda (n-1) } H_{\pi_\beta}(\pi_0) + \frac{\beta \hat{C}}{\Dot{c}}\lambda\]
which completes the proof.
\end{proof}
\begin{proof}[Proof of Corollary \ref{W2 rate}]
Since $\pi$ satisfies an LSI with constant $C_{LSI}$ using Talagrand inequality there holds
\[W_2(\mathcal{L}(\theta^\lambda_n),\pi_\beta)\leq \sqrt{2C_{LSI} H_{\pi_\beta}(\rho_n)}.\]
Using Corollary \ref{H-rate} yields the result.
\end{proof}
\section{Proof of result for excess risk problem}
In this Section we present two Lemmas that bound the terms $T_1$ and $T_2$ in \eqref{eq-optimization problem} and lead to Corollary  \ref{cor-excess risk}. The following have been rigorously proved in \citet{TUSLA},\citet{lim2021non}. We provide only the details that are specific for our work.
\begin{lemma}
    Let Let Assumption \aref{ass-pol lip}, \aref{ass-2dissip} hold. Then,
    \[\E [u(\bar{\theta}^\lambda_n)]-\E_{\pi_\beta}[u(x)]\leq C_1 W_2\left(\mathcal{L}\left(\Bar{\theta}^\lambda_n\right),\pi_\beta\right)\]
    where $C_1$ is given explicitly in the proof.
\end{lemma}

\begin{proof}
    Following the proof in \citet{TUSLA}, Lemma 8, one obtains
    \[\E [u(\bar{\theta}^\lambda_n)]-\E_{\pi_\beta}[u(x)]\leq C_1 W_2\left(\mathcal{L}\left(\Bar{\theta}^\lambda_n\right),\pi_\beta\right).\]
    where \[C_1=\left(\frac{a_1}{l'+1} \sqrt{\mathbb{E}\left|\theta_0\right|^{2 l}+C_{l}^{\prime}}+\frac{a_1}{l+1} \sqrt{\sigma_{2 l'}}+r_2\right) \]
    where $a_1,r_2$ depend on the coefficients in assumption \aref{ass-pol lip}.
and $\sigma_{2l'}$ is the ${2l'}$ moment of $\pi_\beta.$
    Using Theorem 3.4 in \citet{aida1994moment} one obtains \[\sigma_{2l'}\leq \left( \sqrt{\E_{\pi_\beta}[|x|^2]}+C'C_{LSI} (2l'2-2)\right)^p \]
The second moment can be bounded using the dissipativity condition, as by Lemma 3 in \cite{raginsky} the second moment of the Langevin SDE is bounded by  $\frac{b+\frac{d}{\beta}}{a}.$
Using the fact that the Langevin SDE converges to $\pi_\beta$ as $t\rightarrow \infty$ in $W_2$ distance, which implies converge of the respective second moments, one deduces that
\[\E_{\pi_\beta}[|x|^2]=\lim_{t\rightarrow \infty} \E |X_t|^2\leq \frac{b+\frac{d}{\beta}}{a}.\]
\end{proof}
\begin{lemma}
    Let $u^*$ the global minimum of $u$. Then, under Assumptions \aref{ass-pol lip}, \aref{ass-2dissip} there exists $C_2>0$ such that
    \[\E_{\pi_\beta} [u(x)]-u^* \leq \frac{d}{2\beta} \log (C_2 \beta) +\frac{\log 2}{\beta}. \]
\end{lemma}
\begin{proof}
See \citet{lim2021non}, Lemma 4.9.
\end{proof}
% \section{Establishing a Log-Sobolev inequality under Assumptions \aref{ass-derivbound}-\aref{ass-initial cond} and \bref{ass11}-\bref{ass-Morse}}
% \section{Establishing new isoperimetric inequalities}
\section{Proof of Theorem \ref{Poincare constant general}}
Throughout this section we will assume that \aref{ass-derivbound}, \aref{ass-pol lip} and \bref{ass-deltau}-\bref{ass-Morse} and \eqref{eq-limitcond} are satisfied, unless otherwise specified. 
\begin{lemma}
  % Let condition \eqref{eq-limitcond} and assumption \bref{ass-Morse} hold. 
  There exists $N\in \mathbb{N}$ such that the function $u$ has $N$ critical points.
\end{lemma}
\begin{proof}
    Let $C$ the set of critical points of $u$. By condition \eqref{eq-limitcond}, one deduces that $C$ is bounded and since it is a zero set of a continuous function, it is also closed, thus compact.
    Assume that the set of critical points are infinite and can be described as $C=\cup_{i\in I} \{x_i\}$.
    Then, by the inverse function theorem (which holds locally due to \bref{ass-Morse}), there exists $\delta_i$ such that $B(x_i,\delta_i) $ doesn't contain another critical point.\\
    Since $C\subset  \cup_{i\in I}B(x_i,\delta_i) $ by compactness there exists a finite $J \subset I$ such that
    \[C\subset \cup_{i\in J}B(x_i,\delta_i) \] which essentially leads to \[ |C| \leq |J|<\infty.\]
\end{proof}
By \eqref{eq-limitcond} there holds $C\subset B(0,R).$
% Let $K_2$, $K_3$ the gradient and Hessian Lispchitz constants on the ball $B(0,2R_0)$,
Let $K_{2,R}$ the Lipschitz constant of the gradient of $u$ on $B(0,2R)$ and $K_{3,R}:=\max\{\frac{2R}{l^*},K'_{3,R}\}$ and $K'_{3,R}$ is the Lipschitz constant of the Hessian of $u$ on $B(0,2R).$
Recall that $S$ is the set of critical points in $\mathbb{R}^d$ that are not local minimizers  and $\mathcal{C}=\{x^*\}\cup S$ where $x^*$ is the local minimizer. We define $AC_1=\{x\in B_R: d(x,C)\geq \frac{l^*}{K_{3,R}}\}$ and
$AC_{1,1}=\{x\in B_R: d(x,C)>\frac{l^*}{K_{3,R}}\}$.
Since the distance function is continuous and $\Bar{B}_R$ compact this is a compact set. As a result, the function $\frac{|h(x)|}{d(x,C)}$ obtains an
infimum on $AC_1$. Since away from critical points $h$ is non-zero
then,
\[C_{1,1}:=\inf_{AC{1,1}} \frac{|h(x)|}{d(x,C)}\geq \inf_{AC_1}\frac{|h(x)|}{d(x,C)}>0.\]
and $C_c=\min\{C_{1,1},\frac{l^*}{2}\}.$
For $A\geq \sqrt{32} \max\{2\sqrt{K_{2,R}d}/C_c,\frac{l^*}{8},\frac{8}{l^*}\}$ we define:
\begin{equation}
    \begin{aligned}
    U&=\{x\in \mathbb{R}^d : |x-x^*|^2< \frac{A^2}{\beta}\}\\
    E_r&=\{x\in \mathbb{R}^d: d(x,S)^2<\frac{A^2}{\beta}\}\\
    E_{4r}&=\{x\in \mathbb{R}^d: d(x,S)^2< 16\frac{A^2}{\beta}\}\\
    AC&=\{x\in \mathbb{R}^d: d(x,C)^2\geq \frac{A^2}{\beta}\}.
    \end{aligned}
\end{equation}
 Since $C\subset B(0,R)$, for $\beta\geq \frac{A^2}{R^2}$,
$U\subset B(0,2R)$ and $E_{4r}\subset B(0,2R)$.\\
Immediately  by assumption \aref{ass-derivbound} one has that $u$ is $K_{1,R}$-Lipschitz, $h$ $K_{2,R}$-Lipschitz and
$\nabla^2u$ $K_{3,R}$- Lipschitz on $U$ and $E_{4r}$ with constants $\mathcal{O}(R^{2l+1}).$\\
Since the  set of critical points is finite, we assume that there exists $\epsilon_0$ such that \[|x-y|\geq \epsilon_0 \quad \forall x,y \in C.\]
Then, if $\beta \geq 4A^2/\epsilon_0^2$, then
\[(AC \cup U)^c=E_r\]
The restriction for $\beta$ throughout the proof is \begin{equation}\label{eq-betamin}
    \beta\geq \beta_{\min}:=\max\{16A^2/\epsilon_0^2,2 K_{3,R} A^2/l^*,\frac{9C'd}{C_H^2}+4C'd,\frac{A^2}{R^2},(8 K_{3,R} A^3)^2\}
\end{equation}
The proof strategy of this section is sketched as follows. First we deduce a local Poincare inequality around the unique local minimum. Using this result  we are able to use a suitable Lyapunov function to infer a useful inequality on the set $AC\cup U$, i.e around the local minimum or away from  critical points.
Using some escape-time arguments and some results from PDE theory we manage to build Lyapunov functions for the sets around the saddle points and local maximizers. Finally, by building a suitable function to handle the boundary of the sets of interest we are able to connect all these results to infer a Poincare inequality in the whole space which is independent of the dimension and temperature.
\subsection{Poincaré on U}
Let $\pi_U$ the restricted measure on $U$.
If $\beta\geq 2 K_{3,R} A^2/l^*$ we can prove a Poincaré inequality on $U$.
\begin{lemma}
    Let $\pi_U$ the restricted probability measure $\pi_\beta$ on $U$. Then, if $\beta\geq 2 K_{3,R} A^2/l^*$ $\pi_U$ satisfies a Poincare inequality with cosntant $(\kappa_U)^{-1}:=\frac{2}{l^*}$
\end{lemma}
\begin{proof}
Using the inequality \[|\bar{\lambda}_i (A+B)-\bar{\lambda}_i(A)|\leq ||B||_2\]
for $A=\nabla u^2(x^*)$ $B=\nabla^2 u(x)-\nabla^2u(x^*)$ one obtains
\[\lambda_{min} u^2(x)=\lambda_{min}(A+B)\geq \lambda_{\min}(A)-|B|=
\lambda_{min} \nabla^2 u(x^*)-||\nabla^2 u(x)-\nabla^2u(x^*)||\]
Using the fact that $\nabla^2 u(x^*)\geq l^*$ and the $K_{3,R}$-Lipschitz continuity of $\nabla^2 u$ one deduces
\begin{equation}
\lambda_{min}(\nabla^2 u(x))\geq l^*-K_{3,R} |x-x^*|^2\geq l^*-K_{3,R}\frac{A^2}{\beta}\geq \frac{l^*}{2} \quad \forall x\in U.
\end{equation}
which leads to Poincaré inequality on $U$ with constant $\kappa_U=\frac{l^*}{2}$ via the Bakry-Emery Theorem, i.e Theorem \ref{Bakry-Emery}.
\end{proof}
\subsection[s]{Functional inequality on { $AC\cup U$}}
For the rest of the Section we will use an important Lemma.
\begin{lemma}\label{lemma-gradHess}
Assume that there exists $E\subset \mathbb{R}^d$ and $K_E>0$ such that
\[||\nabla^2u(x)-\nabla^2u(y)||\leq K_E |x-y| \quad \forall x,y \in E.\]
There holds,
\[ \left|h(x)-h(y)-\nabla^2 u(y)\left(x-y\right)\right|\leq \frac{K_E}{2} |x-y|^2 \quad \forall x,y \in E.\]
\end{lemma}
\begin{proof}
Fix $x,y \in E.$
    Let $g_{x,y}(t)=h(tx +(1-t)y)$. It is easy to see that
    \[g'_{x,y}(t)= \nabla^2 u(tx+(1-t)y) (x-y)\]
    and $h(x)=g_{x,y}(1)$ and $h(y)=g_{x,y}(0).$
    Putting all together,
    \[\begin{aligned}
        \left|h(x)-h(y)-\nabla^2 u(y)\left(x-y\right)\right|&=\left|g_{x,y}(1)-g_{x,y}(0)+g'_{x,y}(0)\right|\\&=\left|\int_0^1 g'_{x,y}(t)-g'_{x,y}(0) dt \right|\\&\leq
        \int_0^1 \left|g'_{x,y}(t)-g'_{x,y}(0)\right| dt
        \\&\leq \int_0^1||\nabla^2 u(tx+(1-t)y)-\nabla^2 u(y) || |x-y|dt
        \\&\leq \int_0^1 K_E t |x-y|^2 dt
        \\&\leq \frac{K_E}{2} |x-y|^2
    \end{aligned}\]
    where the last steps were derived by the global Lipschitz continuity of $\nabla^2 u$ on $E$.
\end{proof}
For the rest of the section we shall prove some results using the Lyapunov function \begin{equation}\label{eq-lyapdef}
    W:=exp(\frac{\beta u}{2}).
\end{equation}
% \subsubsection{Lyapunov condition on $AC\cap (B(0,R))^c$}
\begin{lemma}
    For $\beta \geq  \frac{9C'd}{C_H^2}+4C'd,$ there holds
    \[\frac{LW}{W}\leq -C'd \quad \forall x\in AC\cap (B(0,R))^c.\]
\end{lemma}
\begin{proof}
By the definition of $W$ in \eqref{eq-lyapdef}, one notices that since $\nabla W=\frac{\beta}{2} W \nabla u,$
\[\langle\nabla W,\nabla u\rangle= \frac{\beta}{2} W \langle \nabla u,\nabla u\rangle=\frac{\beta}{2} W |h(x)|^2.\]
In addition, \[\Delta W=W\left(\frac{\beta^2}{4} |h(x)|^2 +\frac{\beta }{2 } \Delta u \right). \]
Recalling that $LW:=\frac{1}{\beta} \Delta W -\langle \nabla W,\nabla u\rangle$
  one deduces that \[LW/W=\frac{1}{2}\Delta u -\frac{\beta}{4} |h(x)|^2.\]
    By our assumption \bref{ass-deltau} and \eqref{eq-limitcond} one has that
\[ \Delta u\leq 2C'd(1+|h(x)|^2)\]  and \[|h(x)|\geq C_H\]
    so for $\beta\geq \frac{9C'd}{C_H^2}+4C'd,$ one notices that
    \begin{equation}\label{eq-lyapoutball}
        LW/W \leq -C'd
    \end{equation}
\end{proof}
% \subsubsection{Lyapunov condition on $AC\cap B(0,R)$}
\begin{lemma}\label{lemmaAC}
  There holds,  \[LW/W\leq -\min\{C',\frac{1}{2}K_{2,R}\}d \quad \forall x \in AC.\]
\end{lemma}
\begin{proof}
Recall that \begin{equation}\label{eq-C11}
    |h(x)|\geq C_{1,1} d(X,c) \quad \forall x\in AC_{1,1}.
\end{equation}
Let $x\in B(0,R)\cap AC_{1,1}^c.$ There exists $y\in C$ such that \[|x-y|= d(x,C).\]
Since $x,y \in B(0,R)$ using Lemma \ref{lemma-gradHess} for $E:=B(0,R)$ and $K_E:=K_{3,R}$ one obtains
\begin{equation}
\begin{aligned}
|h(x)|&\geq  |\nabla^2u(y)(x-y)|- |h(x)-\nabla^2u(y)(x-y)|\\&\geq |\nabla^2u(y)(x-y)|- \frac{K_{3,R}}{2}|x-y|^2
\end{aligned}
\end{equation}
Let $\{u_i\}_{i=1}^d$ be the orthonormal eigenvectors of $\nabla^2 u(y).$ Then, $x-y=\sum_{i=1}^d c_i u_i$ for some $c_i.$
Writing \[\begin{aligned}
    |\nabla^2 u(y) (x-y)|^2&=|\nabla^2u(y)\sum_{i=1}^d c_i u_i|^2\\&=|\sum_{i=1}^d c_i\bar{\lambda}_i u_i|^2\\&=\sum_{i=1}^d c_i^2 \bar{\lambda}_i^2 \\&\geq (l^*)^2\sum_{i=1}^d c_i^2
    \\&=(l^*)^2 |x-y|^2
\end{aligned}\]
which leads to 
\begin{equation}\label{eq-dxc2}
    |h(x)|\geq l^* |x-y| -\frac{K_{3,R}}{2} |x-y|^2\geq \frac{l^*}{2}d(x,C) \quad \forall x\in B(0,R)\cap AC_{1,1}^c.
\end{equation}
Combining \eqref{eq-C11}, \eqref{eq-dxc2} one deduces that 
\begin{equation}\label{eq-gradistball}
    |h(x)|\geq C_{c}d(x,C) \quad \forall x \in B(0,R)
\end{equation}
where $C_{c}=\min\{C_{1,1},\frac{l^*}{2}\}.$\\
 Using the fact that $\nabla u$ is $K_{2,R}$-Lipschitz implies a that \[||\nabla^2 u(x)||\leq K_{2,R}, \quad \forall x \in B(0,R).\]
 Using this and \eqref{eq-gradistball} one deduces $\forall x \in B(0,R)\cap AC,$
\begin{equation}
\begin{aligned}
LW&=(\frac{1}{2}\Delta u(x)-\frac{\beta}{4}|h(x)|^2)W\\&\leq
    (\frac{1}{2}K_{2,R}d -\frac{\beta}{4} C_c^2 d(x,C)^2) W
    \\&\leq (\frac{1}{2}K_{2,R}d -\frac{1}{4}C_c^2 A^2)W
\end{aligned}
\end{equation}
which yields
\begin{equation}\label{eq-ACBR}
    LW\leq -\frac{1}{2}K_{2,R}d W \quad \forall x\in AC\cap B(0,R),
\end{equation}
which coupled with \eqref{eq-lyapoutball} yields
\[LW/W\leq -\min\{C',\frac{1}{2}K_{2,R}\}d \quad \forall x \in AC.\]
\end{proof}

\begin{lemma}\label{eq-lyap1}
    There holds \[\frac{LW}{W}\leq -\theta  + b_0 \mathds{1}_U \quad \forall x \in AC \cup U\]
    where $\theta=\min\{C',\frac{1}{2}K_{2,R}\}d$
    and $b_0=(K_{2,R}+\min\{C',\frac{1}{2}K_{2,R}\})d$
   \end{lemma}
   \begin{proof}
Noticing that for $x \in U$, since $U\subset B(0,2R),$ \[\frac{LW}{W}\leq\frac{1}{2}\Delta u(x)\leq \frac{1}{2}K_{2,R}d,\] using 
Lemma \ref{lemmaAC} there follows
\[ \frac{LW}{W}\leq -\min\{C',\frac{1}{2}K_{2,R}\}d  + (K_{2,R}+\min\{C',\frac{1}{2}K_{2,R}\})d \mathds{1}_U \quad \forall x\in AC\cup U.\]

\end{proof}

\begin{proposition}\label{poincaretype}
    For all $g \in H^1(\mathbb{R}^d)$ such that $g=0$ on $\partial (AC\cup U)$, there holds
\[ \int_\acu g^2 d\pi_{\beta}\leq \frac{1}{\theta}\frac{1}{\beta} \int_{\mathbb{R}^d} |\nabla g|^2 d\pi_{\beta} +\frac{b_0}{\theta}\int_U g^2 d\pi_{\beta}.\]
\end{proposition}
\begin{proof}
Due to Lemma \ref{eq-lyap1} one deduces
\begin{equation}\label{eq-poingen}
     \begin{aligned}
        \int_\acu g^2 d\pi_{\beta} &\leq \frac{1}{\theta}\int_\acu-\frac{LW}{W}g^2 d\pi_{\beta} +\frac{b_0}{\theta} \int_U g^2 d\pi_{\beta}
    \end{aligned}
\end{equation}
% We will use the fact that $\acu=\left( \cup_{y\in S}B(y,r_{A,\beta})\right)^c$ to bound the first term of the right hand side.
Let $Z_\pi=\int_{\mathbb{R}^d} e^{-\beta u(x)}dx$. 
   One notices that,
   \[ \begin{aligned}
       \frac{1}{\theta}\int_{\acu}-\frac{LW(x)}{W(x)}g^2(x) d\pi_{\beta}(x) &= -\frac{1}{\theta}\frac{1}{\beta} Z_\pi^{-1}\int_{\acu} \Delta W(x)  (g^2(x) e^{-\beta u(x)}/W(x))dx \\&+ \frac{1}{\theta} \int_{\acu} \frac{g^2(x)}{W(x)}\langle  \nabla W(x),\nabla u(x)\rangle d\pi_{\beta}(x)
   \end{aligned}\]
   Since $(AC\cup U)^c=\cup_{y\in S} B(y,\frac{A}{\sqrt{\beta}})$ where the balls are disjoint it has piecewise smooth boundary so the divergence theorem can be applied.
   Using the divergence theorem for the first term one obtains
   \begin{equation}\label{eq-calforball}
       \begin{aligned}
        \frac{1}{\theta}\int_{\acu}-\frac{LW(x)}{W(x)}g^2(x) d\pi_{\beta}(x) &=-
    \frac{1}{\theta} \frac{1}{\beta}   \int_{\partial(\acu)}\frac{g^2(x)}{W(x)} \frac{\partial W}{\partial n}(x) d\pi_{\beta}(x) \\&+ \frac{1}{\theta}\frac{1}{\beta}\int_{\acu} \langle \nabla W(x), \nabla  (g^2 Z_\pi^{-1} e^{-\beta u}/W)(x) \rangle dx  \\&+
    \frac{1}{\theta}\int_{\acu} \frac{g^2(x)}{W(x)}\langle  \nabla W(x),\nabla u(x)\rangle d\pi_{\beta}(x)
    \\&=-
    \frac{1}{\theta} \frac{1}{\beta}   \int_{\partial(\acu)}\frac{g^2(x)}{W(x)} \frac{\partial W}{\partial n}(x) d\pi_{\beta}(x) \\&+ \frac{1}{\theta}\frac{1}{\beta}\int_{\acu} \langle \nabla W(x), \nabla  (g^2/W)(x) \rangle d\pi_\beta(x) \\&-\frac{1}{\theta} \int_{\acu} \frac{g^2(x)}{W(x)}\langle  \nabla W(x),\nabla u(x)\rangle d\pi_{\beta}(x)
    \\&+\frac{1}{\theta}\int_{\acu} \frac{g^2(x)}{W(x)}\langle  \nabla W(x),\nabla u(x)\rangle d\pi_{\beta}(x)
    \\&=-\frac{1}{\theta} \frac{1}{\beta}   \int_{\partial(\acu)}\frac{g^2(x)}{W(x)} \frac{\partial W}{\partial n}(x) d\pi_{\beta}(x) \\&+ \frac{1}{\theta}\frac{1}{\beta}\int_{\acu} \langle \nabla W(x), \nabla  (g^2/W)(x) \rangle d\pi_\beta(x)
    \end{aligned}
   \end{equation}
   Using the fact that $g=0$ on the boundary of $\acu$
   the first term vanishes so after further calculations one obtains
   \begin{equation}\label{eq-argum}
       \begin{aligned}
       \frac{1}{\theta}\int_{\acu}-\frac{LW}{W}g^2 d\pi_{\beta} &=\frac{1}{\theta}\frac{1}{\beta} \int_\acu \langle \nabla W, \nabla \left( \frac{g^2}{W}\right)\rangle d\pi_{\beta}\\&=\frac{1}{\theta}\frac{1}{\beta}  \int_\acu \left( 2 \frac{g}{W} \langle \nabla W,\nabla g\rangle-\frac{g^2}{W^2}|\nabla W|^2\right) d\pi_{\beta}
       \\&\leq\frac{1}{\theta}\frac{1}{\beta} \int_\acu \left( |\nabla g|^2 +\frac{g^2}{W^2}|\nabla W|^2-\frac{g^2}{W^2}|\nabla W|^2\right) d\pi_{\beta}
       \\&= \frac{1}{\theta}\frac{1}{\beta} \int_\acu |\nabla g|^2 d\pi_{\beta} 
       \\&\leq \frac{1}{\theta}\frac{1}{\beta}\int_{\mathbb{R}^d} |\nabla g|^2 d\pi_{\beta}.
   \end{aligned}
   \end{equation}
   Combining \eqref{eq-poingen} and \eqref{eq-argum} completes the proof.
\end{proof}
\subsubsection[Creating a Lyapunov function near Saddle points]{Creating a Lyapunov function on $(AC\cup U)^c$}
To explore the remaining space, a new Lyapunov function needs to be constructed. This function will be based on the probabilistic representation of solution to Dirichlet PDE problems.
\begin{proposition}[\citet{freidlin1985functional}, Theorem 2.1 and Remark 2 and 3 p.127-130]\label{prop-solution}
    Let $E$ a  bounded set with smooth boundary.
    Let $Z_t$ be the solution of the Langevin SDE \[dZ_t=-h(Z_t)dt+\sqrt{\frac{2}{\beta}}dB_t\]
    with initial condition $Z_0.$
    Define $\tau_{E^c}=\inf\{t\geq 0: Z_t \notin E\}.$
    Let $k>0.$
    Assume that \[\sup_{x\in E }\E_x \exp{(2k \tau_{E^c})}:=\sup_{x\in E }\E \left[\exp{(2k \tau_{E^c})}\big|Z_0=x\right]<\infty,  \quad \forall x \in E\]
    Then, the function $W_y(x)=\E_x(exp{(k\tau_{E^c)}})$
    solves the Dirichlet problem
   \[ \begin{array}{cc}
        LW &=-kW \quad x\in E \\
         W&=1 \quad x \in \partial E
    \end{array}\]
\end{proposition}
To this end, we shall create a Lyapunov function bounding the escape time probabilities of a ball near each saddle point.
For the rest of the section we set $r:=r_{A,\beta}:=\frac{A}{\sqrt{\beta}}$ and for $y\in S$ we denote $B^y_{r}=B(y,r_{A,\beta}).$
\begin{lemma}\label{lemma-esc}
Let $y\in S.$
Let $B^y_{4r}:=B(y,4r_{A,\beta}).$
There exist $c_1>0$ such that
$$
\mathbb{P}\left[\tau_{(B^y_{4r})^c} \geq t \mid Z_0=x\right] \leq c_1 e^{-\frac{l^*}{2} t}, \quad \forall t \geq 0, \forall x \in B^y_{4r}
$$
\end{lemma}
\begin{proof}
Let $x\in B^y_{4r}$ and $y\in S$ such that \[|x-y|^2\leq 16 \frac{A^2}{\beta}.\]
Let $g(Z_t)=|Z_t-y|$ and let $Z_t\in B.$
Let $v_y$ be a unit eigenvector of $\nabla^2u(y)$ with respect to its minimum eigenvalue.
We also define $G(x)=|\langle x-y,v_y\rangle |$ and
\[Q(x)=\frac{\langle x-y,v_y\rangle}{|\langle x-y,v_y\rangle |}v_y\quad \forall x\neq y.\]
There also holds that
\begin{equation}\label{eq-G}
    G(x)\leq |x-y||v_y|\leq |x-y|.
\end{equation}
\\
Applying Ito's formula for the process $Y_t=G^2(Z_t)$ one obtains 
\begin{equation}\label{eq-initial}
\begin{aligned}
dY_t=d G^2(Z_t)&=\left[\langle h(Z_t),\nabla G^2(Z_t)\rangle +\frac{1}{\beta}\left(\Delta(G^2(Z_t)) \right)\right]dt\\& +\sqrt{\frac{2}{\beta}}\langle  \nabla G^2(Z_t),d B_t\rangle.
\end{aligned}
\end{equation}
The diffusion coefficient can be written as
\[\sqrt{\frac{2}{\beta}}\langle  \nabla G^2(Z_t),d B_t\rangle=2\sqrt{\frac{2}{\beta}}G(Z_t) dM_t\]
where $M_t$ is 1-dimensional continuous local martingale given by \[M_t=\int_0^t\langle Q(Z_t)\mathds{1}_{Z_s\neq y},dB_s\rangle\]
Since $v_y$ is a unit vector it is easy to see that
\begin{equation}\label{eq-quadvar}
    [M]_t=\E \int_0^t \mathds{1}^2_{Z_s\neq y} ds=\int_0^t P(Z_s\neq y) ds=t
\end{equation}
by the Levy characterization, $M$ is a Wiener Process so
the SDE \eqref{eq-initial} can be rewritten as a 1-dimensional SDE:
\begin{equation}\label{eq-rewritten}
    dY_t=\left[\langle h(Z_t),\nabla G^2(Z_t)\rangle +\frac{1}{\beta}\left(\Delta(G^2(Z_t)) \right)\right]dt +2\sqrt{\frac{2}{\beta}}G(Z_t)dB_t.
\end{equation}
We define the stopped process $X_t:=Z_{t\wedge \tau_{(B^y_{4r})^c}}$.
Let $H(x)=\nabla^2 u(y) (x-y)$.
One notices that since $v_y$ is a unit eigenvector corresponding to the smallest eigenvalue of $\nabla^2 u(y),$
\begin{equation}
\begin{aligned}
\langle H(X_t),\nabla G^2(X_t)\rangle &=2 \langle \nabla^2 u(y) (X_t-y),v_y\rangle \langle X_t-y,v_y\rangle \\&=2\langle X_t-y,\nabla^2 u(y) v_y \rangle {\langle X_t-y,v_y\rangle }\\
&=2 \bar{\lambda}_{min}\left(\nabla^2 u(y)\right) |\langle X_t-y,v_y\rangle|^2
\\&\leq -2l^*  G^2(X_t)
\end{aligned}
\end{equation}

Recall that due to Lemma \ref{lemma-gradHess} and the fact that $X_t\in B(0,2R)$ and $y\in B(0,2R)$, $h(y)=0,$ there holds
\begin{equation}
  |\nabla G^2(X_t)| |H(X_t)-h(X_t)|\leq \frac{K_{3,R}}{2}|X_t-y|^2|\nabla G^2(X_t)|\leq K_{3,R} |X_t-y|^3\leq K_{3,R}8\frac{A^3}{\beta^\frac{3}{2}}.
\end{equation}
so since $\beta \geq (8 K_{3,R} A^3)^2$ \begin{equation}
    | \langle h(X_t)-H(X_t),\nabla  G^2(X_t)\rangle |\leq  \frac{1}{\beta}
\end{equation}
which leads to
\begin{equation}
\begin{aligned}
 -\langle h(X_t),\nabla G^2(X_t)\rangle &=- \langle H(X_t),\nabla G^2(X_t)\rangle +\langle h(X_t)-H(X_t),\nabla G^2(X_t)\rangle \\&\geq  2l^* G^2(X_t)-\frac{1}{\beta} .
\end{aligned}
\end{equation}
In addition, since $\frac{\partial G^2}{\partial x_i}(x)=2\langle x-y,v_y\rangle (v_y)_i $
Then \[\Delta G^2(x)=\sum_{i=1}^d \frac{\partial^2 G^2}{\partial x_i^2}(x)=2 \sum_{i=1}^d(v_y)^2_i=2\]

Set \[F(Y_t,t)=-\langle h(X_t),\nabla G^2(Z_t)+\frac{1}{\beta}\left( \Delta G^2(Z_t) \right)=-\langle h(Z_t),\nabla G^2(Z_t)\rangle+\frac{2}{\beta}.\]
In additon, set \[\bar{F}(x,t)=2l^*x+\frac{1}{\beta}\]
and $\sigma(x)=2\sqrt{2\beta^{-1}}\sqrt{x}.$
Let $\bar{Y}_t$ the solution to the 1- dimensional SDE \begin{equation}
    d\bar{Y}_t=\bar{F}(\bar{Y}_t)dt +\sigma(\bar{Y}_t)dB_t
    \end{equation} with initial condition $\bar{Y}_0=G(x)$.
We also can write \eqref{eq-initial} as
\begin{equation}
    dY_t=F(Y_t)dt+\sigma(Y_t)dB_t
\end{equation}
with initial condition $Y_0=G(x)$. The following properties hold:
\begin{itemize}
    \item  $\bar{F}$ is Lipschitz
    \item \[\int_0^T \sigma^2(Y_t)dt + \int_0^T \sigma^2(\bar{Y}_t)dt <\infty \quad \text{a.s}\]
    \item \[\int_0^T F(Y_t,t)dt + \int_0^T \bar{F}(\bar{Y}_t,t)dt <\infty \quad \text{a.s}\]
    \item \[|\sigma(Y_t)-\sigma(\bar{Y}_t)|\leq 4\sqrt{\beta^{-1}}|Y_t-\bar{Y}_t|^\frac{1}{2}\]
    \item \[F(Y_t,t)\geq \bar{F}(Y_t,t) \quad \forall t\leq \tau_{(B^y_{4r})^c}\]
    \item \[Y_0=\bar{Y}_0.\]
\end{itemize}
Using Corollary \ref{cor-comparison} one obtains
\begin{equation}\label{eq-compary}
    Y_{t\wedge \tau_{(B^y_{4r})^c}}\geq \bar{Y}_{t\wedge \tau_{(B^y_{4r})^c}} \quad \text{a.s}.
\end{equation}
% Since $\bar{Y}$ is a CIR process by Lemma D.7 in \citet{li2020riemannian} there holds \begin{equation}
%     P\left(\bar{Y}_t\leq r\right)= r Ce^{-l^* t/2} \quad \forall r>0.
% \end{equation}
Now we are ready to compute the probability of staying in $B^y_{4r}$.
\begin{equation}
    \begin{aligned}
     P\left[\tau_{(B^y_{4r})^c}\geq t\big|Z_0=x\right]&\leq P\left[\{sup_{s\in[0,t]} |Z_s-y|^2\leq  16 r^2\}\cap \{\tau_{(B^y_{4r})^c}\geq t\}\big|Z_0=x\right]
     \\&\leq P\left[\{Y_t\leq 16 r^2\}\cap \{\tau_{(B^y_{4r})^c}\geq t\}\big|Z_0=x\right]  \quad (\text{using \eqref{eq-G}} )
     \\&\leq P\left[\{Y_{t\wedge\tau_{(B^y_{4r})^c}}\leq 16 r^2\}\cap \{\tau_{(B^y_{4r})^c}\geq t\} \big| Y_0=G^2(x)\right]
     \\& \leq P\left[\bar{Y}_t\leq 16 r^2\big |\bar{Y}_0=G^2(x)\right] \quad ( \text{using \eqref{eq-compary}})
     \\&\leq  \int _0 ^{16r^2} f_{{\bar{Y}},G^2(x)}(t,y) dy.
     \\&\leq e^{\lambda_1 16 r^2 }\int _0 ^{16r^2} e^{-\lambda_1 y} f_{{\bar{Y}},G^2(x)}(t,y) dy
     \\&\leq C e^{-\frac{l^*}{2}t} \quad (\text{Proposition } \ref{prop-CIR})
    \end{aligned}
\end{equation}
for some $C>0$ independent of $x.$
\end{proof}
% \begin{theorem}[\citet{evans2010partial} Section 6.3 , Theorem 4]\label{theo-smoothness}
%     Let $\bar{L}$ a uniformly elliptic operator with the divergence form \[\bar{L} u=-\sum_{i, j=1}^n\left(a^{i j}(x) u_{x_i}\right)_{x_j}+\sum_{i=1}^n b^i(x) u_{x_i}-c(x) u\] and $E$ a bounded open set. Let $a^{i j}\in \mathcal{C}^2(\bar{E})$ and $b^i,c \in \mathcal{L}^\infty(E)$.\\
%     Suppose that $\partial E \in \mathcal{C}^2$ and $f \in $
%     Let $W$ a weak solution to the Dirichlet problem
% \begin{align}
%         \bar{L}u &=f \quad x\in E\\
%          u&=0 \quad x \in \partial E .
%     \end{align}
%     Then, $u\in H^2 ({E}).$ Consequently, \[Lu= f\quad \text{a.e in } E.\]
% \end{theorem}
\begin{corollary}\label{escape lyapunov}
Let $\kappa:=\frac{l^*}{8}.$ Let $y\in S.$
    There exists a function $\bar{W}_y \in \mathcal{C}^{2}({B^y_{4r}})$ such that $\bar{W}_y\geq 1 \quad \forall x \in B_{4r}^y$ and
    \[L\Bar{W}_y=-\kappa\Bar{W}_y \quad \text{in} \quad B^y_{4r}.\]
\end{corollary}
\begin{proof}
Using the equivalent definition of subexponential random variables in Theorem \ref{subexponential} and Lemma \ref{lemma-esc} one obtains  \[\sup_{x\in B^y_{4r}}\E_x \exp{\left(2\kappa \tau_{(B^y_{4r})^c}\right)}<\infty.\] Using Proposition \ref{prop-solution} the function
    $\bar{W}_y(x)=\E_x\left(exp{\left(\kappa\tau_{(B^y_{4r})^c}\right)}\right)$
    solves the Dirichlet problem
   \[ \begin{array}{cc}
        LW &=-\kappa W \quad x \in B^y_{4r} \\
         W&=1 \quad x \in \partial B^y_{4r}
    \end{array}\]
    which yields the result.
\end{proof}
\subsubsection{Proving a Poincare inequality on the whole space}
We have now proved a functional ineuquality on $AC\cup U$ and a Lyapunov result in Corollary \ref{escape lyapunov}. Before we prove the Poincaré inequality on the whole space we shall use the following construction.
%\begin{proposition}[\citet{li2020riemannian},Proposition 14] \label{boundary extension} For all $r>\widetilde{r}>0$, let us define the following open neighbourhoods of saddle points
%\end{proposition}
\begin{lemma}
 For all $r_2>r_1>0$ and $\epsilon>0$, there exists a smooth non-decreasing function $\psi: \mathbb{R} \rightarrow[0,1]$ such that
$$
\psi(x)= \begin{cases}0, & x \leq r_1, \\ \text { Increasing }, & x \in\left(r_1, r_2\right), \\ 1, & x \geq r_2,\end{cases}
$$
\[\text { and that }\left\|\psi^{\prime}\right\|_{\infty} \leq \frac{1}{r_2-r_1}+\epsilon.\]
In addition, by setting $\chi(x)=\psi(d(x,S)),$ $r_1=2r,$ $r_2=3r$, where $r=\frac{A}{\sqrt{\beta}}$,
one obtains $0\leq\chi\leq 1$ is a smooth function, such that
$\nabla \chi=0$ on $E_{r}:=\{x\in \mathbb{R}^d:d(x,S)< r_{A,\beta}\}$  and \[||\Gamma(\chi,\chi)||_\infty \leq \frac{1}{\beta}(||\psi'||_\infty)^2\leq \frac{2}{\beta}\left(\frac{1}{r^2}+\epsilon^2\right)\leq \frac{4}{A^2}\leq \frac{1}{8} \frac{1}{\frac{1}{\kappa} +\frac{1}{\theta}}.\]
\end{lemma}
    \begin{proof}
    For the construction of the function $\psi$,
        see proof of Lemma B.13 in \cite{li2020riemannian}.
    \end{proof}

\begin{proof}[Proof of Theorem \ref{Poincare constant general} ]
 Let $0\leq\chi\leq 1$ be a smooth function such that $\chi=0$ on $E_{2r}$ and $\chi=1$ on $E_{3r}^c.$
 We define a Lyapunov function $W_2$ as  \[\begin{array}{ccc}
      W_2(x)&=&\bar{W}_y(x), \quad  x\in B^y_{4r} \\
      W_2(x)&=&1, \quad \text{elsewhere}.
 \end{array}\]
 Note that $E_{4r}=\{x\in \mathbb{R}^d:d(x,S)< 4 r_{A,\beta}\}=\cup_{y\in S} B^y_{4r}$ and by the restriction on $\beta$, $B^y_{4r}$ are disjoint. Then, for $f\in H^1(\mathbb{R}^d),$
   \begin{equation}\label{eq-mainstartpoincare}
\begin{aligned}
\int_{\mathbb{R}^d} f^2 d\pi_{\beta} &=\int_{\mathbb{R}^d}(f(1-\chi)+f \chi)^2 d\pi_{\beta} \\
& \leq 2 \int_{\mathbb{R}^d} f^2(1-\chi)^2 d\pi_{\beta}+2 \int_{\mathbb{R}^d} f^2 \chi^2 d\pi_{\beta} \\
& \leq \frac{2}{\kappa} \sum_{y\in S} \int_{B^y_{4r}} \frac{-L \bar{W}_y}{\bar{W}_y} f^2(1-\chi)^2 d\pi_{\beta}+2 \int_{AC\cup U} f^2 \chi^2 d\pi_{\beta}  \\
& \leq \frac{2}{\kappa} \sum_{y\in S} \int_{B^y_{4r}} \frac{-L W_2}{W_2} f^2(1-\chi)^2 d\pi_{\beta}+2 \int_{AC\cup U} f^2 \chi^2 d\pi_{\beta}  \\
& = \frac{2}{\kappa} \int_{E_{4r}} \frac{-L W_2}{W_2} f^2(1-\chi)^2 d\pi_{\beta}+2 \int_{AC\cup U} f^2 \chi^2 d\pi_{\beta}.
\end{aligned}
\end{equation}
For the first term since $f(1-\chi)$ is zero on the boundary of $E_{4r}$ (which is piecewise smooth) using arguments as in \eqref{eq-argum}  there holds
\[\frac{2}{\kappa} \int_{E_{4r}} \frac{-L W_2}{W_2} f^2(1-\chi)^2 d\pi_{\beta}\leq \frac{2}{\kappa} \int_{\mathbb{R}^d} \Gamma (f(1-\chi),f(1-\chi))d\pi_{\beta}.\]
Using \[\Gamma(fg,fg)\leq 2\left(f^2 \Gamma(g,g) +g^2 \Gamma(f,f)\right),\]
yields
\begin{equation}\label{eq-firstterm}
    \frac{2}{\kappa} \int_{E_{4r}} \frac{-L W_2}{W_2} f^2(1-\chi)^2 d\pi_{\beta}\leq \frac{4}{\kappa}\frac{1}{\beta} \int_{\mathbb{R}^d} |\nabla f|^2 d\pi_{\beta} + \frac{4}{\kappa} ||\Gamma(1-\chi,1-\chi)||_\infty \int_{\mathbb{R}^d}f^2 d\pi_{\beta}
\end{equation}
In addition, since $(AC\cup U)^c=E_r$,  $g=f\chi$ vanishes on the boundary of $AC\cup U$ thus, using Lemma \ref{poincaretype} one deduces
\begin{equation}\label{eq-secondterm2}
    \begin{aligned}
    2\int_\acu (f\chi)^2 d\pi_{\beta}&\leq \frac{2}{\theta}\frac{1}{\beta} \int_{\mathbb{R}^d} |\nabla f\chi|^2 d\pi_{\beta} +\frac{2b_0}{\theta}\int_U (f\chi)^2 d\pi_{\beta}
    \\&\leq \frac{4}{\theta}\frac{1}{\beta} \int_{\mathbb{R}^d} |\nabla f|^2 d\pi_{\beta} +\frac{4}{\theta}||\Gamma(\chi,\chi)||_\infty \int_{\mathbb{R}^d}f^2d\pi_{\beta} +\frac{2b_0}{\theta}\int_U f^2 d\pi_{\beta}
\end{aligned} 
\end{equation}
Bringing \eqref{eq-mainstartpoincare},\eqref{eq-firstterm}, \eqref{eq-secondterm2} together one obtains
\[\begin{aligned}
    \int_{\mathbb{R}^d} f^2 d\pi_{\beta} &\leq \frac{4}{\beta}\left(\frac{1}{\kappa}+\frac{1}{\theta}\right)\int_{\mathbb{R}^d} |\nabla f |^2 d\pi_{\beta} +4 ||\Gamma(\chi,\chi)||_\infty\left(\frac{1}{\kappa}+\frac{1}{\theta}\right)\int_{\mathbb{R}^d} f^2 d\pi_{\beta} \\&+ \frac{2b_0}{\theta} \int_{U} f^2 d\pi_{\beta}
\end{aligned}\]
Using the fact that $4 ||\Gamma(\chi,\chi)||_\infty\left(\frac{1}{\kappa}+\frac{1}{\theta}\right)\leq \frac{1}{2}$,
one obtains 
\[\int_{\mathbb{R}^d} f^2 d\pi_{\beta} \leq \frac{8}{\beta}\left(\frac{1}{\kappa}+\frac{1}{\theta}\right)\int_{\mathbb{R}^d} |\nabla f |^2 d\pi_{\beta} + \frac{4 b_0}{\theta}\int_{U} f^2 d\pi_{\beta}\]
In order to conclude the proof we will make use of the fact that a Poincare inequality on $U$ has already been proved.
Let $\Tilde{f} \in H^1(\mathbb{R}^d)$. Setting $f=\Tilde{f}-c_{\Tilde{f}}$ where $c_{\Tilde{f}}=\frac{1}{\pi(U)} \int_U \Tilde{f}d\pi_{\beta}$.
Using the fact that $Var(\Tilde{f})=Var_\pi({f}+c_{\Tilde{f}})\leq \int_{\mathbb{R}^d} f^2 d\pi_\beta,  $ one deduces
\begin{equation}\label{eq-Poincarefinal}
\begin{aligned}
Var_\pi (\Tilde{f})\leq \int_{\mathbb{R}^d} f^2 d\pi_{\beta} \leq  \frac{8}{\beta}\left(\frac{1}{\kappa}+\frac{1}{\theta}\right)\int_{\mathbb{R}^d} |\nabla \Tilde{f} |^2 d\pi_{\beta} + \frac{4 b_0}{\theta}\int_{U} f^2 d\pi_{\beta}
\end{aligned}
\end{equation}
Using the fact that $\pi$ satisfies a PI on $U$ and  $\int_U f d\pi_{\beta}=0,$ the last term is bounded i.e
\[\frac{4 b_0}{\theta}\int_{U} f^2 d\pi_{\beta}\leq \frac{4 b_0}{\theta}\frac{1}{\beta}\frac{1}{\kappa_U} \int_U |\nabla {f} |^2  d\pi_{\beta} \leq  \frac{4 b_0}{\theta}\frac{1}{\beta}\frac{1}{\kappa_U} \int_{\mathbb{R}^d} |\nabla \Tilde{f} |^2  d\pi_{\beta}. \]
Inserting this into \eqref{eq-Poincarefinal} one deduces \[Var_\pi (\Tilde{f})\leq \frac{1}{\beta}\left( \frac{4b_0}{\theta}\frac{1}{\kappa_U} + \frac{8}{\kappa}+\frac{8}{\theta}\right) \int_{\mathbb{R}^d} |\nabla \Tilde{f}|^2 d\pi_{\beta}\]
which completes our proof.
\end{proof}
\section{Proof of Corollary \ref{theo-LSI constant}}
\begin{remark}\label{remarkpoinc}
 Let Assumptions \aref{ass-derivbound}-\aref{ass-initial cond} and \bref{ass11}-\bref{ass-Morse}  hold.
Then, for $\beta \geq \mathcal{O}(d^6)$, 
 $\pi_\beta$ satisfies a Poincaré inequality with constant $C_p$ independent of the dimension.
\end{remark}
\begin{proof}
Note that \aref{ass-2dissip} implies \eqref{eq-limitcond}
    with constant $R=\max\{4\sqrt{\frac{b}{a}}, 1\}$ , $c_H=\frac{a}{2}.$
   Using the previous restriction for $\beta$ one obtains the result.
\end{proof}
\begin{lemma}\label{LSI-prelim}
Let $c_0=\frac{a}{4}$ and $\gamma=\frac{a(d/\beta+ b)}{2}$ Then, under \aref{ass-2dissip} there exists a Lyapunov function $\hat{W} \in C^2$ such that \[L\hat{W}\leq \beta \left(-c_0|x|^2 +\gamma\right) \hat{W} \]
\end{lemma}
\begin{proof}
Setting $\hat{W}=e^{\beta \frac{a}{4} |x|^2}$
    \[L\hat{W}=\left(\frac{ad}{2}\beta +\frac{a^2}{4} \beta|x|^2-\frac{a \beta}{2} \langle h(x),x\rangle \right ) \hat{W}\]
    Using the fact that \[\langle h(x),x\rangle \geq a|x|^2-b\] the result follows easily.
\end{proof}
\begin{theorem}[\cite{menz2014poincare}, Theorem 3.15]
    Suppose that there exists a $\mathcal{C}^2$ function $W\geq 1$ such that:
    \begin{enumerate}
        \item There exist $c_0,\gamma>0$ such that\[\frac{LW}{W}(x)\leq \beta \left(-c_0 |x|^2 +\gamma\right)\quad\forall x \in \mathbb{R}^d.\]
        % \item There holds \[\int_{\mathbb{R}^d} (-LW) f d\pi_\beta =\frac{1}{\beta}\int_{\mathbb{R}^d} \langle \nabla W,\nabla f\rangle d\pi_\beta \quad \forall f \in H^1(\mathbb{R}^d).\]
        \item $\pi_\beta$ satisfies a Poincare inequality with constant ${\rho}.$ 
        \item There exists $K_0\in \mathbb{R}^d$ such that \[\nabla^2 u(x)\geq -K_0 I_d \quad \forall x\in \mathbb{R}^d.\]
    \end{enumerate}
    Then, $\pi_\beta$ satisfies an LSI with constant
    \[\frac{1}{C_{LSI}}\leq 2\beta \sqrt{\frac{1}{c_0}\left(\frac{1}{2}+\frac{ \gamma + c_0\E_{\pi_\beta} [|x|^2]}{\beta \rho}\right)}+\frac{\beta^2 K_0}{c_0} + \frac{\beta}{\rho} K_0(\gamma +c_0\E_{\pi_\beta}[|x|^2])+\frac{2}{\rho}. \]
\end{theorem}
\begin{proof}[\textbf{Proof of Corollary \ref{theo-LSI constant}}]
    Using the Lyapunov function $\hat{W}$ in Lemma \ref{LSI-prelim} one finds that condition 1 in the previous theorem is satisfied. Since a Poincare inequality has already been proved (see Remark \ref{remarkpoinc}) and condition 3 holds by our assumptions, the previous theorem can be applied.
\end{proof}

\section{Motivating examples}
\subsection{Example 1: Example satisfying a Poincare constant under our novel assumptions}
Let $g: \mathbb{R}^2 \rightarrow \mathbb{R}$ such that
$g(x, y)=\frac{1}{3} x^3+y^2+2 x y-6 x-3 y+4.$\\
One calculates $g_x=x^2+2y-6$ and $g_y=2y+2x-3.$
The set of critical points are $C=\{(-1,\frac{5}{2}), (3,-\frac{3}{2})\}.$
The Hessian of the function is given by $g_{xx}=2x$, $g_{xy}=2$, $g_{yy}=2$. \\
It is easy to see that that $(-1,\frac{5}{2})$
is a non-degenerate saddle point and $(3,-\frac{3}{2})$ is a non-degenerate local minimum. It is also easy to see that \[\lim_{|x|\rightarrow \infty} |\nabla g(x) |=\infty,\] the gradient local Lipschitz conditions are satisfied which is also true for the polynomial growth of the higher derivatives. Therefore the assumptions of Theorem \ref{Poincare constant general} are satisfied, so for $\beta$ large enough, the measure $\mu=e^{\beta g}/C$ satisfies PI independent of $\beta$.
\subsection{Example 2: Improving the Log-Sobolev constant for our algorithm}
Let $u: \mathbb{R}^2 \rightarrow \mathbb{R}$ such that
$u(x,y)=\frac{1}{6} (|x|-1)^6\mathds{1}_{|x|\geq 1}-x^2-4x +\frac{y^2}{2}-y$
Then, $u_x=\frac{x}{|x|}(|x|-1)^5 \mathds{1}_{|x|\geq 1} -2x -4$ $u_y=y-1$, $u_{xy}=0$.
The function has a unique critical point at $(x,y)=(2.5567,1)$ which is a non-degenerate local minimum.
 It is easy to see the dissipativity condition as well as a lower bound on the smallest eigenvalue are easily satisfied as well as the rest of the  assumptions of Theorem \ref{theo-LSI constant}, so for $\beta$ large enough the measure $\pi_\beta$ satisfies Log-Sobolev inequality with polynomial dependence on $\beta$.\\ 
\section{Improving the dimension dependence of LSI constant under a convexity at infinity assumption-Proof of Theorem \ref{lemma-comparisonying}}\label{example-compYing}

\newcommand{\tu}{\tilde{u}}
\newcommand{\hu}{\hat{u}}
\newcommand{\tV}{\tilde{V}}
Before establishing our result we prove an intermediate result which is the extension of Lemma 4 in \citet{ma2019sampling} to the Local Lipschitz case.
\begin{lemma}\label{ma-ext}
   Suppose that  $u\in \mathcal{C}^2$ such that \eqref{ass-pol lip} is satisfied. In addition, we assume that there exist $m,R>0$ such that $u$ is $m$-strongly convex on $\bar{B}(0,R)^c$ in the sense of Definition \ref{def-convalt}.\\
   Then, $\pi_\beta$ satisfies LSI with constant independent of the dimension.
\end{lemma}
\begin{proof}
The ultimate goal of the proof is to construct a strongly convex and smooth function $\hu$ with Hessian that exists everywhere.
In this proof we extend the work of \citet{ma2019sampling} in the local-Lipschitz gradient case.
The proof strategy is the same but the main difference is that we will work with the local Lipschitz constant in each bounded set, instead of a global one. 
% In addition, in the relevant Lemma \ref{flammanionconnection} we have shown that our assumption implies the convexity assumption in \cite{ma2019sampling}  and Lemma 4 in \citet{ma2019sampling} is substituted by  Lemma \ref{flammlemma4adap}.\\
Let us now present the proof.\\
By Assumption \aref{ass-pol lip}, $\nabla u$ is Lipchitz on $\bar{B}(0,3R)$ with Lipschitz constant  $L_R:=L(1+6R)^l.$ We define \[\tu(x)=u(x)-\frac{m}{4}|x|^2\]
  For $x,y \in \bar{B}(0,\frac{3}{2}R),$ \[|\nabla \tu(x)-\nabla\tu(x-y)|\leq (L_R+\frac{m}{2}) |y|\] and
  \begin{equation}\label{eq-tulip}
      \tu(x-y)-\tu(x)\leq \int_0^1 |\nabla \tu(tx+(1-t)(x-y))||y|dt \leq C_R|y|.
  \end{equation}
  where $C_{{R}}:=\left(L_R+\frac{m}{2}+|\nabla u(0)|\right).$
Let \begin{equation}\label{eq-delta}
    \delta\leq \frac{1}{2}\min\{\frac{m}{C_R}, \frac{R}{6}\}.
\end{equation}
Let $\Omega=\bar{B}(0,R)^c$
    and $V$ be the convex extension of $\tu$ on $\mathbb{R}^d$ given as \begin{equation}\label{eq-convext}
        V(x)=\inf _{\substack{\left\{x_i\right\} \subset \Omega,\left\{\lambda_i \mid \sum_i \lambda_i=1\right\}, \\ \text { s.t., } \sum_i \lambda_i x_i=x}}\left\{\sum_{i=1}^l \lambda_i \tilde{u}\left(x_i\right)\right\}, \quad \forall x \in \mathbb{R}^d
    \end{equation}

     By Lemma \ref{flammlemma4adap},
    \[V(x)=\tu(x)\quad \forall x\in \bar{B}(0,R)^c\] and
    \begin{equation}\label{eq-comp1}
        \inf_{|x|=R} \tu(x)\leq V(x)\leq \sup_{|x|=R} \tu(x) \quad \forall x \in \bar{B}(0,R).
    \end{equation}
    Let $\phi\geq 0$ be a mollifier supported on $\bar{B}(0,\delta)$ such that $\int \phi(y)dy=1.$
    Let $\tV$ be a smoothing of $V$ on $\bar{B}(0,\frac{4}{3}R)$ given as
    \[\tV(x)=\int V(x-y) \phi(y)dy.\]
    Then, $\tV$ is convex and smooth on $\mathbb{R}^d.$ It is also $\frac{m}{2}$ strongly convex in $\bar{B}(0,\frac{4}{3}R+\delta)^c.$
    In addition, by definition, \[\inf_{x\in \bar{B}(0,\frac{4R}{3}+\delta)} V(x) \leq \tV(x) \leq \sup_{x\in \bar{B}(0,\frac{4R}{3}+\delta)} V(x),\quad \forall |x|\leq \frac{4}{3}R\]
    so using \eqref{eq-comp1},
    \begin{equation}\label{eq-comp2}
        \inf_{\bar{B}(0,\frac{4R}{3}+\delta)/\bar{B}(0,R)}\tu(x) \leq \tV(x)\leq \sup_{\bar{B}(0,\frac{4R}{3}+\delta)/\bar{B}(0,R)} \tu(x) \quad \forall |x|\leq \frac{4}{3}R.
    \end{equation}
    Finally, we construct the auxiliary function $\hu(x)$ :
$$
\hu(x)-\frac{m}{4}\left|x\right|^2=\left\{\begin{array}{l}
\tu(x), \quad|x|\geq\frac{3}{2} R \\
\alpha(x) \tilde{u}(x)+(1-\alpha(x)) \tilde{V}(x), \quad \frac{4}{3} R<|x|<\frac{3}{2} R, \\
\tilde{V}(x), \quad|x|\leq \frac{4}{3} R
\end{array}\right.
$$
where $\alpha(x)=-\frac{1}{2} \cos \left(\frac{36 \pi}{17} \frac{|x|^2}{R^2}-\frac{64 \pi}{17}\right)+\frac{1}{2}$. Here we know that $\tu(x)$ is $\frac{m}{2}$-strongly convex and smooth in $\mathbb{R}^d \backslash \mathbb{B}(0, R) $ and $ \tilde{V}(x)$ is $\frac{m}{2}$-strongly convex and smooth in ${B}\left(0, \frac{4}{3} R\right)^c$.\\ For $\frac{4}{3} R<|x|<\frac{3}{2} R$,
\begin{equation}\label{eq-Hessianbound}
\begin{aligned}
 \nabla^2\left(\hu(x)-\frac{m}{4}\left|x\right|^2\right) 
& =\nabla^2 \tu(x)+\nabla^2((1-\alpha(x))(\tilde{V}(x)-\tu(x))) \\
& =\alpha(x) \nabla^2 \tu(x)+(1-\alpha(x)) \nabla^2 \tilde{V}(x) \\
& -\nabla^2 \alpha(x)(\tilde{V}(x)-\tu(x))-2 \nabla \alpha(x)(\nabla \tilde{V}(x)-\nabla \tu(x))^T \\
& \geq \frac{m}{2} I_d-\nabla^2 \alpha(x)(\tilde{V}(x)-\tu(x))-2 \nabla \alpha(x)(\nabla \tilde{V}(x)-\nabla \tu(x))^T .
\end{aligned}
\end{equation}
Let $x$ such that $\frac{4}{3}R\leq |x|\leq \frac{3}{2}R.$
For $|y|<\delta<\frac{R}{6}$, one notices that \[|x-y|\geq \left||x|-|y|\right|\geq R\]
which leads to
\[V(x-y)=\tu(x-y)\] and \[\nabla V(x-y)=\nabla \tu(x-y). \]
Then, 
\begin{equation}\label{eq-b1}
    |\nabla \tV(x)-\nabla \tu(x)|\leq \int |\nabla \tu(x-y)-\nabla \tu(x)| \phi(y)dy\leq (L_R+\frac{m}{4})\delta \quad \forall \frac{4}{3}R\leq |x|\leq\frac{3}{2}R
\end{equation}
and 
\begin{equation}\label{eq-b2}
\begin{aligned}
     \tilde V(x)-\tu(x)&=\int (\tu (x-y)-\tu(x))\phi(y) dy\\&\leq \int \left(\int_0^1 |\nabla \tu \left((1-t)(x-y)+ty\right)| |y|dt\right)\phi(y) dy\\&\leq C_{L_R}\delta \quad \forall \frac{4}{3}R\leq |x|\leq\frac{3}{2}R.
\end{aligned}
\end{equation}
Applying \eqref{eq-b1}, \eqref{eq-b2} to \eqref{eq-Hessianbound} one obtains that for $\frac{4}{3}R<|x|<\frac{3}{2}R,$
\[\nabla^2\left( \hu-\frac{m}{4}|x|^2\right)\geq \frac{m}{2}I_d -C_R\delta\geq 0.\]
and by the definition of $\hu$ one concludes that it is $\mathcal{C}^1(\mathbb{R}^d)$, with Hessian that exists everywhere and $\frac{m}{2}$ strongly convex.
By the Bakry-Emery theorem the measure $\hat{\pi}_\beta:=\frac{e^{-\beta \hu(x)}}{\int_{\mathbb{R}^d}e^{-\beta \hu(x)}dx}$ satisfies an LSI with constant $(C_{LSI})^{-1}=\frac{2}{m}.$
In addition, using the definition of $\hu$ one deduces that
\[\hu(x)-\frac{m}{4}|x|^2-\tu(x)=0 \quad for |x|>\frac{3}{2}R.\]
In addition using the bound in \eqref{eq-b2},
\[|\hu(x)-\frac{m}{4}|x|^2-\tu(x)|\leq |1-a(x)||\tV-\tu|\leq \frac{C_{L_R}R}{6} \quad \forall \frac{4}{3}R<|x|<\frac{3}{2}R\]
and due to \eqref{eq-comp2} and \eqref{eq-tulip}, \[|\hu(x)-\frac{m}{4}|x|^2-\tu(x)|\leq \sup_{\bar{B}(0,\frac{4}{3}R)} \tu(x)-\inf_{\bar{B}(0,\frac{4}{3}R)}\tu(x)\leq \frac{4}{3}R C_R \quad \forall |x|\leq \frac{4}{3}R.\]
Putting all together since $\hu(x)-u(x)=\hu(x)-\frac{m}{4}|x|^2-\tilde{u}$,
\[||\hu-u||_\infty\leq \max\{\frac{4}{3}R C_R,\frac{C_{L_R}R}{6}\} \]
which leads to
\begin{equation}\label{eq-oscilation}
\sup(\beta \hu-\beta u)-\inf(\beta \hu-\beta u)\leq 2C'_{R}\beta
\end{equation}
where $C'_{R}:=\max\{\frac{4}{3}R C_R,\frac{C_{L_R}R}{6}\}$.
As a result, using the Hooley-Stook perturbation principle \cite[p.1184]{holley1986logarithmic}, since the $\hat{\pi_\beta}$ satisfies
LSI with constant $\frac{m}{2}$ then, $\pi_\beta$ satisfies an LSI with constant
\[(C_{LSI})^{-1}=e^{\beta 2 C'_{R}}\frac{2}{m}.\]

\end{proof}
\begin{proof}[Proof of Theorem \ref{lemma-comparisonying}.]
    By Lemma \ref{flammanionconnection} $u$ is $m$- strongly convex on $\bar{B}(0,R_m)^c.$ Applying Lemma \ref{ma-ext} completes the result.
\end{proof}
\begin{proof}[Proof of Corollary \ref{cor-reg}]
By Lemma \ref{lemma-reg} it is easy to see that $u$ satisfies the assumptions of Theorem \ref{lemma-comparisonying}. Thus, the result immediately follows.     
\end{proof}
\appendix
\section{Appendix}
\begin{proposition}[\citet{lamberton2011introduction}, Proposition 6.2.4]\label{prop-CIR}
Let $X_t$ the solution of
\[dX_t=(a-bX_t) dt +\sigma \sqrt{Y_t}dB_t \]
with initial condtion $X_0=x$
Then, for any $\lambda>0$ there holds
\[\E(e^{-\lambda_1 X_t}) =\frac{1}{(2 \lambda_1 L+1)^{2 a / \sigma^2}} \exp \left(-\frac{\lambda_1 L \zeta}{2 \lambda_1 L+1}\right)\]
where $L=\left(\sigma^2 / 4 b\right)\left(1-e^{-b t}\right)$ \text { and } $\zeta=4 x b /\left(\sigma^2\left(e^{b t}-1\right)\right).$
Setting $\sigma:=2\sqrt{\frac{2}{\beta}}$, $b=-2l^*$, $a=\frac{1}{\beta}$, $X_0=G(x)$ and $\lambda_1=\frac{-2b}{\sigma^2}$ ,
\[\E(e^{-\lambda_1 X_t})\leq (e^{b t})^\frac{2a}{\sigma^2} =(e^{-2l^* t})^\frac{1}{4}=e^{-\frac{l^*}{2}t}.\]
\end{proposition}
\begin{theorem}
    [\citet{wainwright2019high}, Theorem 2.13]\label{subexponential}
    Let a random variable $X$ and suppose that there exist $c_1,c_2>0$ such that 
    \[P(|X|\geq t)\leq c_1 e^{-c_2 t} \quad \forall t >0.\]
    Then, there holds 
    \[\E e^{\frac{c_2}{2}|X|}<\infty.\]
    \end{theorem}
    \begin{proposition}  [\citet{pardoux2014stochastic}, Proposition 2.33]
    Let $F,\bar{F},\sigma:\mathbb{R}\rightarrow \mathbb{R}.$
        Let $Y_t$ $\bar{Y}_t$ the solutions to the one dimensional SDEs with filtration $\mathcal{F}$
        \[Y_t= y_0+\int_0^t F(Y_s)ds +\int_0^t \langle \sigma(Y_s) dB_s\rangle\]
        and \[\bar{Y}_t=y_0++\int_0^t \bar{F}(\bar{Y}_s)ds +\int_0^t \langle \sigma(\bar{Y}_s) dB_s\rangle,\]
        $\Bar{F}$ is $L$- Lipschitz, $\sigma$ is $\frac{1}{2}$- Hoelder continuous and there holds for all $T\geq 0$
        \[\int_0^T |\sigma(Y_t)|^2dt+\int_0^T|\sigma(\bar{Y}_t)|^2 dt<\infty \quad \text{a.s}\]
        and
        \[\int_0^T |F(Y_t)| dt +\int_0^T |\bar{F}(\bar{Y}_t)| dt <\infty \quad \text{a.s}. \]
        Then, 
        \[(\bar{Y}_t-Y_t)^+=\int_0^t (\bar{F}(\bar{Y}_s)-F(Y_s))\theta(\Bar{Y}_s-Y_s)ds +\int_0^t \langle (\sigma(\bar{Y}_s)-\sigma(Y_s))\mathds{1}_{\bar{Y}_s-Y_s>0},dB_s\rangle\]
        where \[\theta(x)= \begin{cases}0, & \text { if } x<0 \\ \frac{1}{2}, & \text { if } x=0 \\ 1, & \text { if } x>0.\end{cases}\]
    \end{proposition}
    \begin{corollary}\label{cor-comparison}
        Let $Y,\bar{Y}$ as in the previous proposition and further suppose that 
        \[\int_0^t\E|Y_s| +\E | \bar{Y}_s| <\infty.\]
        Let $\tau$ a stopping time adapted to $\mathcal{F}_t.$
        Assume that if $t\leq \tau$, \[\bar{F}(Y_t)\leq F(Y_t)\]
        Then,
        \[Y_{\tau \wedge t}\geq \bar{Y}_{\tau \wedge t}\]
    \end{corollary}
    \begin{proof}
        Using the previous Proposition one obtains
        \begin{equation}\label{eq-plus}
            (\bar{Y}_t-Y_t)^+=\int_0^t (\bar{F}(\bar{Y}_s)-F(Y_s))\theta(\Bar{Y}_s-Y_s)ds +\int_0^t \langle (\sigma(\bar{Y}_s)-\sigma({Y}_s))\mathds{1}_{\bar{Y}_s-Y_s>0},dB_s\rangle
        \end{equation}
        Using the fact that $\sigma$ is Hoelder continuous and the first moments of $Y,\bar{Y}$ are finite then,
        \[\E \int_0^{t\wedge \tau}(\sigma(Y_s)-\sigma(\bar{Y}_s))^2\mathds{1}_{\bar{Y}_s-Y_s>0} dt<\infty \] so after applying \eqref{eq-plus} for the time $t\wedge \tau$ and taking expectations,
        the expectation of the stochastic integral vanishes which leads to
        \[\begin{aligned}
            \E (\bar{Y}_{t\wedge \tau}-Y_{t\wedge \tau})^+&=\E \int_0^{t\wedge \tau} (\bar{F}(\bar{Y}_s)-F(Y_s))\theta(\bar{Y}_s-{Y}_s)ds
            \\&=\E \int_0^t \mathds{1}_{(0,\tau)}(s) (\bar{F}(\bar{Y}_s)-\bar{F}(Y_s))\theta(\Bar{Y}_s-Y_s)ds
            \\&+ \E \int_0^t \mathds{1}_{(0,\tau)}(s) (\bar{F}(Y_s)-{F}(Y_s)) \theta(\Bar{Y}_s-Y_s)ds
        \end{aligned}\]
        Using the fact that the second term is non-positive, then using the Lipschitz property of $\bar{F}$ one obtains
        \[\E (\bar{Y}_{t\wedge \tau}-Y_{t\wedge \tau})^+\leq L \E \int_0^t \mathds{1}_{(0,\tau)}(s) (\bar{Y}_s-Y_s)^+ ds\leq L  \int_0^t \E (\bar{Y}_{s\wedge \tau}-Y_{s\wedge \tau})^+ ds.   \]
        Since the right hand side is finite, and $\E(\bar{Y}_{0\wedge \tau}-Y_{0\wedge \tau})^+=0,$
        an application of Grownwall's lemma yields
        \[\E(\bar{Y}_{t\wedge \tau}-Y_{t\wedge \tau})^+=0\]
        which means that
        \[Y_{t\wedge \tau}\geq \bar{Y}_{t\wedge \tau} \quad \text{a.s}\].
        
    \end{proof}

\subsection{Important lemmas for proof of exchanges in derivative and integral}
In this Section, we recall assumptions \aref{ass-derivbound}, \aref{ass-initial cond}.
Throughout the proof we will use generic constants, which need not be specified, as we are only interested in the growth of the different coefficients with respect to the state variable.
This results shall be used to show the finiteness of specific integrals.
Throughout this section many matrix calculus identities will be used.
The interested reader can point to \citet{petersen2008matrix}. 
\newtheorem{Lemma1}[Def1]{Lemma}
\begin{lemma}\label{pt-decay}
Let Assumptions
\aref{ass-derivbound}, \aref{ass-initial cond} hold. Suppose that $\sqrt{\lambda}\leq \frac{1}{2C_1}$ where $C_1:=2a +4L$ where $L$ and $a$ are given in Assumptions \aref{ass-derivbound} and \aref{ass-2dissip}
Then, there exist constants $A,r>0$, independent of $x$, uniform in a small neighbourhood of $t$, such that  \[|\pt(x)|\leq Ae^{-r|x|^2}.\]
\end{lemma}
\begin{proof}
The proof will be done by induction. First of all,
$\hat{\pi}_0$ decays exponentially. We assume the this is also true for $\pkl$ i,e there exist $A,r$ such that
\begin{equation}\label{eq-ass}
    \pkl(x)\leq  Ae^{-r|x|^2}.
\end{equation}
Let \begin{equation}
    \phi(x)=x-(t-k\lambda)h_\lambda(x).
\end{equation}
Let $Z=\phi(x_{k\lambda})=x_{k\lambda}-(t-k\lambda)h_\lambda(x_{k\lambda})$ and let $p(\cdot)$ be its density. Then, $\pt=p*\mu$ where $\mu$ is the distribution of $\mathcal{N}(0,t-\kappa \lambda)$.

Taking derivatives with respect to $x$ yields
\[J_\phi(x)=(1-(t-k\lambda)a)I_d -(t-k\lambda)J_{f\lambda}\]
where $f_\lambda= (h(x)-ax)g_{\lambda}$
and $g_\lambda=\frac{1}{1+\sqrt{\lambda}|x|^{2l}}$.
As a result, for $H:=\nabla^2 u$, one calculates
\[\begin{aligned}
(t-k \lambda)(a I_d +||J_{f_\lambda}||)&\leq \lambda(a + ||(H-aI_d)g_\lambda + \nabla g_\lambda \otimes (h(x)-ax)||
\\&\leq \lambda( a +(||H||+a)g_\lambda + |\nabla g_\lambda(x) ||h(x)-a x|)\\&\leq (2a +L+1) \sqrt{\lambda}\leq \frac{1}{2}.
\end{aligned}\]
which means that $\frac{3}{2}I_d>J_\phi>\frac{1}{2} I_d$ i.e $J_\phi$ is positive semidefinite which implies that $\phi$ is a bi-Lipschitz map i.e 
\[\frac{1}{2}|x-y|\leq |\phi(x)-\phi(y)|\leq \frac{3}{2}|x-y| \quad \forall x,y \in \mathbb{R}^d.\]
Using Theorem A (for $\mathcal{C}^2$ functions) in \cite{gordon1972diffeomorphisms} and the inverse function theorem, one deduces that $\phi$ is bijective and there exists a $\mathcal{C}^2$ inverse map $\phi^{-1} :\mathbb{R}^d\rightarrow \mathbb{R}^d$ such that \[J_{\phi^{-1}}=(J_\phi)^{-1}.\]
By the definition of the Jacobian of $\phi^{-1}$ one deduces that it is also Lipschitz, since $J_\phi$ is lower bounded. 
By the inversion formula it easy to see that
\begin{equation}\label{eq-p}
    p(x)=\frac{\pkl(\phi^{-1}(x))}{det(J_\phi(\phi^{-1}(x))}\leq \frac{\pkl(\phi^{-1}(x))}{\left(\bar{\lambda}_{min}(J_\phi)\right)^d}\leq \pkl(\phi^{-1}(x)){2^d}.
\end{equation}
By assumption \ref{eq-ass} \[\pkl(\phi^{-1}(x))\leq A e^{-r |\phi^{(-1)}(x)|^2}\] and since
\[||x|-|\phi(0)||\leq |\phi(\phi^{-1}(x))-\phi(0)|\leq 2 |\phi^{-1}(x)|\]
one deduces that \begin{equation}\label{eq-p final}
    p(x)\leq A e^{-r |\phi^{-1}(x)|^2}\leq  A' e^{-r''|x|^2}.
\end{equation}
To complete the proof it suffices to notice that
\[\begin{aligned}
    \hat{\pi}_{t}(x)&=p * \mu(x)\\&=\int p(y) \mu(x-y) d y \\&\leq \int A' e^{-r|y|^{2}}(2 \pi(t-k \eta))^{-\frac{d}{2}} e^{-\frac{|x-y|^{2}}{2(t-k \eta)}} d y \\&\leq A'' e^{-r''|x|^{2}}
\end{aligned}\]
which leads to the result by induction.
\end{proof}
\newtheorem{highdivpolgrowth}[Def1]{Lemma}
\begin{lemma}
Let Assumptions
\aref{ass-derivbound}, \aref{ass-initial cond} hold.
Then, 
\[\max\{|| J^{(2)}_\phi(x)||,||J^{(3)}_\phi(x)||\}\leq C(1+|x|^q)  \]
for some $q\in \mathbb{N}.$
\end{lemma}
\begin{proof}
    Since the derivatives of $f(x)=h(x)-ax$ have polynomial growth and the derivatives of $\frac{1}{1+\sqrt{\lambda}|x|^{2l}}$ are always bounded, the result follows easily.
\end{proof}
\newtheorem{gradlog1 growth}[Def1]{Lemma}
\begin{lemma}\label{gradlog1 growth}
Let Assumptions
\aref{ass-derivbound}, \aref{ass-initial cond} hold.
Then, there exist $C$,$q'$>0, independent of $x$, uniform in small neighbourhood of $t$ such that \[|\nabla_x \log \hat{\pi}_t(x)|\leq C(1+|x|^{q'}).\]
\end{lemma}
\begin{proof}
Recall that the function $\phi$ has the following properties:
\begin{itemize}
\item  $\phi,\phi^{-1}: \mathbb{R}^d\rightarrow \mathbb{R}^d$ are $\mathcal{C}^3$ mappings.
    \item $\phi$ is Lipschitz, i.e $J_{\phi}$ is bounded.
    \item $\phi^{-1}$ is Lipschitz, i.e $J_{{\phi}^{-1}}(x)=(J_{\phi})^{-1}(\phi^{-1}(x))$ is bounded.
\end{itemize}
The proof shall be done by induction. We assume that $|\nabla \log \pi_{k\lambda}|$ has polynomial growth, a condition which is true for $|\nabla \log \pi_0|$.
The proof starts with the following decomposition \begin{equation}\label{eq-decop1}
    \nabla \log p(x)=\nabla_x \log (\pi_{k\lambda}(\phi^{-1}(x)))-\nabla_x \log (det(J_{\phi} (\phi^{-1})(x))).
\end{equation}
For the first term,
\begin{equation}\label{eq-21}
    |\nabla_x\log (\pi_{k\lambda}(\phi^{-1}(x)))|\leq ||J _{\phi^{-1}}(x))|| |\nabla \log \pi_{k\lambda} (\phi^{-1}(x))|
\end{equation}
using the fact that $|\nabla \log \pi_{k\lambda}| $ has polynomial growth and that $\phi^{-1}$ is 2-Lipschitz and 2-times differentiable ( as proved in the previous Lemma) and \eqref{eq-21} becomes
\begin{equation}\label{eq-21imp}
|\nabla_x\log (\pi_{k\lambda}(\phi^{-1}(x)))|\leq C (1+|\phi^{-1}(x)|^q)||J_{\phi^{-1}}(x)||\leq C' (1+2|x|^q).
\end{equation}
For the second term, one notices that
\[\log det(J_\phi (\phi^{-1}(x))=\log \prod_{i=1}^d \bar{\lambda}_i (J_\phi(\phi^{-1}(x)))=\sum_{i=1}^d \log \bar{\lambda}_i (J \phi(\phi^{-1}(x))). \]
Taking derivatives yields
\[\begin{aligned}|\frac{\partial}{\partial x_i}\log det(J_{\phi} (\phi^{-1}(x))|&=|tr\left(\frac{\partial}{\partial x_i} (J_\phi (\phi^{-1}(x)))\left(J_\phi(\phi^{-1}(x)))\right)^{-1}\right)|\\&\leq d ||\left(J_\phi(\phi^{-1}(x)))\right)^{-1}||||\frac{\partial}{\partial x_i} (J_\phi (\phi^{-1}(x)))||\\&\leq 4 d||\frac{\partial}{\partial x_i} (J_\phi (\phi^{-1}(x)))||
\end{aligned}\]
so \begin{equation}\label{eq-logdet}
\begin{aligned}
     |\nabla_x \log det(J_\phi(\phi^{-1}(x)))|&\leq 4d \sqrt{\sum_{i=1}^d ||\frac{\partial}{\partial x_i} (J_\phi (\phi^{-1}(x)))||^2}\\&\leq 4d^\frac{3}{2} C(1+|\phi^{-1}(x)|^q)\\&\leq C' d^\frac{3}{2}(1+|x|^q).
\end{aligned}
\end{equation}
Combining \eqref{eq-logdet} and \eqref{eq-21} yields that there exists $q'$ such that \begin{equation}\label{eq-comb}
    |\nabla \log p(x)|\leq C(1+|x|^{q'}).
\end{equation}
Since $\hat{\pi}$ is a convolution of $p$ with the Gaussian density one deduces the following bound:
\begin{equation}\label{eq-conv1}
    \begin{aligned}
\left|\nabla \log (p*\mu)(x)\right|=\left|\nabla_x \log \int p(x-y)\mu(y)dy\right|&\leq \left| \frac{\int \nabla p(x-y)\mu(y)dy}{\int \mu(x-y) p(y) dy}\right|
\\&\leq \frac{\int |\nabla \log p(x-y)| p(x-y) \mu(y) dy}{\int p(x-y) \mu(y) dy}
\\&=\frac{\int |\nabla \log p(y)| p(y) \mu(x-y)dy}{\int p(x-y) \mu(y) dy}
\\&\leq C\frac{\int (1+|y|^{q'})p(y)\mu(x-y)dy}{\int p(y)\mu(x-y)dy}
\end{aligned}
\end{equation}
Let $R>0$ such that $\int_{B(0,R)} p(x) dx\geq \frac{1}{5}.$ Then,
\begin{equation}\label{eq-conv2}
    \begin{aligned}
    \int_{B(0,R)}p(y)\mu(x-y)dy& \geq C_1 e^{-c_1|x|^2} \int_{B(0,R)}p(y) e^{-c_1|y|^2}dy\\&\geq C_1 e^{-c_1|x|^2} \int_{B(0,R)}p(y)e^{-c_1|R|^2}dy\\&\geq C_1'(e^{-c_1|x|^2+c_1R^2}).
\end{aligned}
\end{equation}
In addition is easy to see that due to the decaying tails of $p$ and $\mu$, there exists $K$ such that for $y \in B(0,K(R+|x|))^c $ there holds $\mu(x-y)\leq e^{-cK(R+|x|)^2}\leq e^{-c_1|x|^2-c_1R^2}$ so
\begin{equation}\label{eq-conv3}
\begin{aligned}
    \int_{B(0,K(R+|x|))^c}(1+|y|^{q'}) p(y)\mu(x-y)&\leq e^{-c_1|x|^2-c_1R^2}\int_{B(0,K(R+|x|))^c}(1+|y|^{q'}) p(y)dy\\&\leq C''_1e^{-c_1|x|^2-c_1R^2}.
    \end{aligned}
\end{equation}
Combining \eqref{eq-conv1} \eqref{eq-conv2}, \eqref{eq-conv3} yields
\begin{equation}\label{eq-conv4}
\begin{aligned}
|\nabla  \log \hat{\pi}_t(x)|&=|\nabla \log (p*\mu)|\\&\leq\frac{\int(1+|y|^{q'})p(y)\mu(x-y)dy}{\int p(y) \mu(x-y)dy}\\&\leq\frac{\int_{B(0,K(R+|x|))}(1+|y|^{q'}) p(y)\mu(x-y)dy}{\int p(y) \mu(x-y)dy} \\&+\frac{\int_{B(0,K(R+|x|))^c}(1+|y|^{q'}) p(y)\mu(x-y)dy}{\int_{B(0,R)} p(y) \mu(x-y)dy}\\&\leq (1+(K(R+|x|)^q) +\frac{C''_1}{C_1}.
\end{aligned}
\end{equation}
As a result, $|\nabla \log \hat{\pi}_t|$ has polynomial growth for every $t\in [k\lambda,(k+1)\lambda].$
\end{proof}
\newtheorem{nabla2}[Def1]{Lemma}
\begin{lemma}\label{gradlog2 growth}
Let Assumptions
\aref{ass-derivbound}, \aref{ass-initial cond} hold.
Then, 
there exist $C$,$q''$>0 independent of $x$, uniform in small neighbourhood of $t$, such that \[
||\nabla^2 \log \hat{\pi}_t(x)||\leq C(1+|x|^{q''}) \quad \forall x\in \mathbb{R}^d.\]
\end{lemma}
\begin{proof}
As in the previous lemmas the proof will be inductive. Since $|\nabla^2 \log \hat{\pi}_0|$ the base assumption is satisfied and we assume that there exists $C,k'$ such \begin{equation}\label{eq-assnab2}
    \|\nabla^2 \log \hat{\pi}_{k\lambda}(x)\|\leq C(1+|x|^{k'}).
\end{equation}
Taking derivatives in \eqref{eq-decop1} one obtains
\begin{equation}\label{eq-basenab2}
    \begin{aligned}
    \nabla^2\log p(z)= \nabla^2_z\log \hat{\pi}_{k\lambda}(\phi^{-1}(z))-\nabla^2_z \log det\left(J_\phi)(\phi^{-1}(z))\right).
    \end{aligned}
\end{equation}
To bound $
    |\nabla^2_z\log \hat{\pi}_{k\lambda}(\phi^{-1}(z))|$ one notices that
    \[\begin{aligned}
        \frac{\partial}{\partial z_i}\nabla\log \hat{\pi}_{k\lambda}(\phi^{-1}(z))&= \frac{\partial}{\partial z_i}(J_{\phi^{-1}}(z))\nabla (\log \hat{\pi}_{k\lambda}) (\fin) \\&+J_{\phi^{-1}}(z) \nabla^2(\log \hat{\pi}_{k\lambda})(\fin)  \frac{\partial}{\partial z_i}\phi^{-1}(z).
    \end{aligned} \]
    From the previous lemmas there holds
    \begin{equation}\label{eq-nabla21}
        |\nabla \log\hat{\pi}_{k\lambda}(\phi^{-1}(z))|\leq C (1+|\fin|^k)\leq C'(1+|z|^k)
    \end{equation}
    and by induction hypothesis
    \begin{equation}\label{eq-nabla22}
        ||\nabla^2\log \hat{\pi}_{k\lambda}(\fin)||\leq C(1+|\fin|^{k'})\leq C(1+|z|^{k'}).
    \end{equation}
    In addition, using the formula for the derivative of inverse i.e $\partial (A^{-1})=-A^{-1} \partial A A^{-1}$, writing \[\begin{aligned}
        \frac{\partial}{\partial z_i} (J_{\phi^{-1}})(z)&= \frac{\partial}{\partial z_i} (J_\phi)^{-1}(\fin)\\&=-(J_\phi)^{-1}(\fin) \frac{\partial}{\partial z_i} \left(J_\phi(\fin)\right) (J_\phi)^{-1}(\fin)
    \end{aligned}\] 
    which, since $J^{(2)}_\phi$ has polynomial growth, $(J_{\phi})^{-1}$ and $J_{\phi^{-1}}$ is bounded
    yields
    \begin{equation}\label{eq-nabla23}
     ||\frac{\partial}{\partial z_i}J_{\phi^{-1}}(z)||\leq C(1+|z|^{P})
    \end{equation}
    for some $C>0$, $P\in \mathbb{N}.$
    Combining \eqref{eq-nabla21}, \eqref{eq-nabla22} and \eqref{eq-nabla23} one concludes that
    \begin{equation}\label{eq-nabla24}
        ||\nabla^2_z\log \hat{\pi}_{k\lambda}(\phi^{-1}(z))||\leq C(1+|z|^r).
    \end{equation}
    In order to bound the term $\nabla^2_z\log det (J_{\phi} (\fin))$, recall that
    \begin{equation}
    \begin{aligned}
         \frac{\partial}{\partial z_i} \log det (J_{\phi} (\fin))&=tr\left(\frac{\partial}{\partial z_i} (J_{\phi}(\fin)) (J_{\phi}(\fin))^{-1}\right)\\&=tr\left(\frac{\partial}{\partial z_i} (J_{\phi}(\fin)) J_{\phi^{-1}}(z)\right)
    \end{aligned}
    \end{equation}
    so taking derivatives one obtains
    \begin{equation}
    \begin{aligned}
\left|     \frac{\partial}{\partial z_i\partial z_j}\log det (J_{\phi} (\fin))\right|&\leq\left|tr\left( \frac{\partial}{\partial z_i\partial z_j}(J_{\phi}(\fin)) J_{\phi^{-1}}(z)\right)\right| \\&+\left|tr\left( \frac{\partial}{\partial z_i} (J_{\phi}(\fin)) \frac{\partial}{\partial z_j}J_{\phi^{-1}}(z) \right) \right|
    \end{aligned}
    \end{equation}
    Since this expression contains first,second,third derivatives of $\phi$, first and second derivative of $\phi^{-1}$ which all have polynomial growth there follows that $||\nabla^2_z\log det (J_{\phi} (\fin))||_F$ has polynomial growth. As a result, $||\nabla^2_z\log det (J_{\phi} (\fin))||$ has also polynomial growth.
    Combining this along with \eqref{eq-nabla24} yields
    \begin{equation}
        ||\nabla^2\log p(z)||\leq C(1+|z|^{m}).
    \end{equation}
    For the convolution with the Gaussian,
    \begin{equation}\label{eq-nabla2conv}
    \begin{aligned}
        ||\nabla^2(\log (p*\mu))||&\leq  ||\frac{\nabla^2(p*\mu)}{p*\mu}|| +\frac{(\nabla p*\mu)(\nabla p*\mu)^T}{(p*\mu)^2}  ||\\&\leq ||\frac{\nabla^2(p*\mu)}{p*\mu}|| + \left(\frac{|\nabla p*\mu|}{(p*\mu)} \right) ^2.
    \end{aligned}
    \end{equation}
    Noticing that $\frac{\nabla p*\mu}{p*\mu}=\nabla \log (p*\mu)=\nabla \log \hat{\pi}_t$ has polynomial growth it remains to bound the first term. Writing
    \[\begin{aligned}
        ||\frac{\nabla^2(p*\mu)}{p*\mu}||&=\left \|\frac{\int \nabla \log p(y)\nabla \log \mu(x-y)p(y)\mu(x-y)dy}{\int p(y)\mu(x-y)dy }\right\|\\&\leq C\frac{\int(1+|x|^q+|y|^{q})p(y)\mu(x-y)dy}{\int p(y)\mu(x-y)dy }.
    \end{aligned}\]
    Using the same arguments as in the previous lemma there easily follows that $\frac{\nabla^2(p*\mu)}{p*\mu}$ has polynomial growth.
    Applying this to \eqref{eq-nabla2conv} one concludes that for every $t\in [k\lambda,(k+1)\lambda]$, $||\nabla^2 \log \hat{\pi}_t||$ has polynomial growth which concludes the inductive proof.
    \end{proof}
    \subsection{Auxiliary Lemmas to prove LSI under convexity at infinity condition}
    We provide a definition that extends the notion of convexity to non-convex sets.
    \begin{Def}\label{def-convalt}
        Let $\Omega\subset \mathbb{R}^d$ a non-convex set and $f:\Omega\rightarrow \mathbb{R}$. 
        $f$ is convex if \[f(x) \leq \lambda_1 f\left(x_1\right)+\cdots+\lambda_k f\left(x_k\right)\]
        holds whenever $x_1, \ldots, x_k \in \Omega$ and their convex combination $x=\lambda_1 x_1+\cdots+\lambda_k x_k \in \Omega$.
        % \item  $f$ is locally convex if at any $x \in \Omega$, there is a ball $B$ around $x$, such that the restriction of $f$ to $B \cap \Omega$ is convex.
        % \end{itemize}
    \end{Def}
    \begin{lemma}[\cite{yan2012extension}, Proposition 2.1]\label{straight}
    Let $\Omega\subset \mathbb{R}^d$ open and $\tu:\mathbb{R}^d\rightarrow\mathbb{R}$ satisfying 
    \begin{equation}\label{eq-strline}
         \lambda \tu(x) +(1-\lambda) \tu(y)\geq \tu\left(\lambda x +(1-\lambda) y\right) \quad \forall x,y,\lambda x +(1-\lambda) y \in \Omega, \quad \forall 0\leq\lambda\leq 1.
    \end{equation}
    Then, $\tu$ is convex on $\Omega.$
\end{lemma}
\begin{proof}
    Let $x,x_i \in \Omega$ such that $x_i\neq x$ and \[x:=\sum_{i=1}^n l_i x_i\] where $l_i>0$ and $\sum_{i=1}^n l_i=1.$
    There exist $\epsilon_0>0$ such that \[B(x,\epsilon_0)\subseteq \Omega.\]
    Let $\delta:=\frac{\epsilon_0}{\max_i\{|x-x_i|\}}.$
Setting $y_i:=\delta x_i+(1-\delta)x$ one observes that \[|y_i-x|\leq \delta |x_i-x|\leq \epsilon_0\]
which implies that $y_i\in B(x,\epsilon_0), \quad \forall i=1,2,...n, $ and as a result,
\[y_i\in \Omega \quad \forall i=1,2,...n.\]
Since $x,x_i,y_i \in \Omega$ by \eqref{eq-strline} it is implied that
\begin{equation}\label{eq-baseyi}
    \tu(y_i)\leq \delta \tu(x_i)+(1-\delta)\tu(x).
\end{equation}
Let $x,y \in B(0,\epsilon_0).$ Since $B(0,\epsilon_0)$ is convex, for any $0\leq \lambda \leq 1$, \[(\lambda x +(1-\lambda)y)\in B(0,\epsilon_0),\] so since $B(0,\epsilon_0)\subset \Omega$ one deduces that
\begin{equation}\label{eq-convfortwo}
    \tu(\lambda x+(1-\lambda)y)\leq \lambda \tu (x) +(1-\lambda)\tu(y) \quad \forall x,y \in B(0,\epsilon_0), 0\leq\lambda\leq 1.
\end{equation}
Since $B(0,\epsilon_0)$ is convex, \eqref{eq-convfortwo} implies that $\tu$ is convex on $B(0,\epsilon_0)$.
As a result, since $y_i \in B(0,\epsilon_0)$ , $x\in B(0,\epsilon_0)$ and \[\sum_{i=1}^n l_iy_i = \sum_{i=1}^n l_i \delta x_i + \sum_{i=1}^n l_i (1-\delta)x =x,\] one notices that
\[\begin{aligned}
    \tu(x)&\leq \sum_{i=1}^n l_i \tu(y_i)
    \\&\leq \sum_{i=1}^n \left[l_i\delta \tu(x_i) +(1-\delta)l_i\tu(x)\right]
    \\&=\delta\sum_{i=1}^n l_i \tu(x_i) +(1-\delta) \tu(x)
\end{aligned}\]
where the second step was given by \eqref{eq-baseyi}.
Rearranging, there follows that 
\[\tu(x)\leq \sum_{i=1}^n l_i \tu(x_i),\]
which completes the proof.
\end{proof}
    % \begin{theorem}[\citet{yan2012extension}, Theorem 5.1 ]\label{connectconv}
    %      Suppose $\Omega\subset\mathbb{R}^d$ is a convex set, and $A\subset \mathbb{R}^d$ is a closed set of dimension $\leq \operatorname{dim} \Omega-2$. Then for continuous functions on $\Omega\cap A^c$, the convexity and the local convexity are equivalent.
    % \end{theorem}
    % We will now use this theorem to show that a positive definite Hessian " at infinity" is connected to  Definiton \ref{def-convalt}. 
       \begin{lemma}\label{flammanionconnection}
    Suppose $u:\mathbb{R}^d\rightarrow \mathbb{R},\quad u\in \mathcal{C}^2$ satisfying \[\langle \nabla u(x)-\nabla u(y),x-y\rangle \geq \left(c_1(|x|^{2r}+|y|^{2r})-c_2(|x|^l+|y|^l)-c_3\right)|x-y|^2 \quad \forall x,y \in \mathbb{R}^d\] for some $c_1,c_2,c_3>0$ and $2r>l>0$. Then, for any $m>0$, $u$ is $m$-strongly convex in the sense of Definition \ref{def-convalt} on $\bar{B}(0,R_m)^c,$ 
    where 
    \begin{equation}\label{eq-radiusconv}
        R_m=\left(\frac{2}{c_1}\left(\left(\frac{c_2}{2c_1}\right)^{\frac{1}{2r-l}}+c_3+m\right)\right)^{\frac{1}{2r}}.
    \end{equation}
.
       \end{lemma}
      %  \begin{proof}
      %  By the condition for $\nabla^2 u$ it is easy to see that $\hat{u}$ is locally convex on $B(0,\frac{R}{2})^c.$
      % \\ Let \[A:=\{x\in \mathbb{R}^d: x=(x_1,x_2,\dots,x_{d-2},0,0) \land |x|=\frac{2}{3}R\}.\]
      % One notices that $A$ is a closed set of dimension $\leq d-2$ \\and since $A^c\subset B(0,\frac{R}{2})^c$, $\hat{u}$ is locally convex on $A^c$.\\
      % Using Theorem \ref{connectconv} for the same set $A$ and $\Omega:=\mathbb{R}^d$, one deduces that on $A^c$ local convexity is equivalent to convexity, thus $\hat{u}$ is convex on $A^c$. \\
      % As a result, since $B(0,R)^c\subset A^c$, $\hat{u}$ is also convex on $B(0,R)^c$.
      %      \end{proof}
      \begin{proof}
    Let $m>0$ $\tu:= u-\frac{m}{2}|\cdot|^2.$ It is easy to see that
    \[\langle \nabla \tu(x)-\nabla \tu(y),x-y\rangle\geq \left(c_1(|x|^{2r}+|y|^{2r})-c_2(|x|^l+|y|^l)-c_3-m\right)|x-y|^2 \quad \forall x,y \in \mathbb{R}^d.\]
    Setting $L_0:=\left(\frac{c_2}{2c_1}\right)^{\frac{1}{2r-l}}+c_3+m$ one deduces
    \begin{equation}\label{eq-sconv}
        \langle \nabla \tu(x)-\nabla \tu(y),x-y\rangle\geq \left(\frac{c_1}{2}|x|^{2r}+\frac{c_1}{2}|y|^{2r}-L_0\right)|x-y|^2\quad \forall x,y \in \mathbb{R}^d.
    \end{equation}
    % Let $R:=\left(\frac{2L_0}{c_1}\right)^\frac{1}{2r}$
    Let $x$,$y\in \bar{B}(0,R)^c$ and $\lambda>0$ such that $z:=x+y \in \Bar{B}(0,R)^c.$
    Writing 
    \begin{equation}\label{eq-sconv2}
        \begin{aligned}
     \hspace{-20pt}   \lambda \tu(x)-\lambda \tu(z)-\lambda \langle \nabla \tu(z),x-z\rangle &= \lambda \int_0^1 \langle \nabla \tu(tx +(1-t)z)-\nabla \tu(z),x-z\rangle dt
        \\&= \lambda \int_0^1 \frac{1}{t}\langle \nabla \tu(tx +(1-t)z)-\nabla \tu(z),t(x-z)\rangle dt
        \\&\geq \lambda \int_0^1 \frac{1}{t}\left(\frac{c_1}{2}|tx+(1-t)z|^{2r}+\frac{c_1}{2}|z|^{2r}-L_0\right)t^2|x-z|^2dt
        \\&\geq \lambda \int_0^1 t(\frac{c_1}{2}|z|^{2r}-L_0)|x-z|^2 dt 
        \\&=\lambda(\frac{c_1}{2}|z|^{2r}-L_0)\frac{|x-z|^2}{2}
        \\&\geq 0
    \end{aligned}
    \end{equation}
    where the first inequality is derived by \eqref{eq-sconv} and the last step is given by the fact that $|z|\geq R.$
    Using the same arguments for $y$ in place of $x$ and $(1-\lambda)$ in place of $\lambda$ one deduces that 
    \begin{equation}\label{eq-sconv3}
        (1-\lambda)\tu(y)-(1-\lambda)\tu(z)-(1-\lambda) \langle \nabla \tu(z),x-z\rangle\geq 0.
    \end{equation}
    Adding \eqref{eq-sconv2} and \eqref{eq-sconv3} one obtains
    \[\begin{aligned}
    \lambda \tu(x) +(1-\lambda) \tu(y)-\tu\left(\lambda x +(1-\lambda) y\right)&=
        \lambda \tu(x)+(1-\lambda)\tu(y)-\tu(z)\\&\geq \langle \nabla \tu(z), \lambda x -\lambda z +(1-\lambda)y-(1-\lambda)z\rangle\\&= \langle \nabla \tu(z),\lambda x +(1-\lambda)y -z\rangle\\&=0
        \quad \forall x,y,\lambda x +(1-\lambda) y \in \bar{B}(0,R)^c.
    \end{aligned}\]
    Applying Lemma \ref{straight} yields the result.
\end{proof}
      \begin{lemma}\label{flammlemma4adap}
    Suppose $u:\mathbb{R}^d\rightarrow \mathbb{R},$ $ u\in \mathcal{C}^2$ satisfying \aref{ass-pol lip}.
           Then for $\Omega:=\bar{B}(0,R)^c,$ and $\tilde{u}=u(x)-\frac{m}{2}|x|^2,$
        \[  V(x)=\inf _{\substack{\left\{x_i\right\} \subset \Omega,\left\{\lambda_i \mid \sum_i \lambda_i=1\right\}, \\ \text { s.t., } \sum_i \lambda_i x_i=x}}\left\{\sum_{i=1}^l \lambda_i \tilde{u}\left(x_i\right)\right\}, \quad \forall x \in \mathbb{R}^d\]
        is a convex extension of $\tilde{u}$ such that $V$ is convex in $\mathbb{R}^d$ and 
        \[V(x)=\tu(x) \quad \forall x \in \Omega.\]
        In addition, there holds
        \begin{equation}\label{V-tu bound}
        \inf_{|x|=R} \tilde{u}(x)\leq V(x)\leq \sup_{|x|=R} \tu(x) \quad \forall x \in \bar{B}(0,R).
        \end{equation}
      \end{lemma}
      \begin{proof}
        See  \citet{ma2019sampling}, Lemma 4.
      \end{proof}
  % \begin{proof}
  %     Recall that in Lemma \ref{flammanionconnection} it is shown that $\tilde{u}$ is convex on $\Omega.$ In addition, using Theorem 1 of \citet{peters1986convex} and the fact that $\tu$ is convex, it can be seen that $V$ is the convex  extension of $\tu$ to the convex hull of $\Omega$ (which is $\mathbb{R}^d$) and thus,
  %     \[V(x)=\tu(x) \quad \forall x \in \bar{B}(0,R)^c.\]
  %     Finally, since $\tu$ is convex on $\bar{B}(0,R)^c$, the rest of the proof is identical to the one in Lemma 4 of \citet{ma2019sampling}.
  % \end{proof}
  \begin{lemma}\label{lemma-reg}
    Let $g:\mathbb{R}^d\rightarrow\mathbb{R}$ such that \[|\nabla g(x)-\nabla g(y)|\leq L (1+|x|^l+|y|^l)|x-y|\quad \forall x,y\in \mathbb{R}^d.\] Let $r>\frac{l}{2}$ and $\eta>0$. Then the function $u$ given by $u(x):=g(x)+\eta |x|^{2r+2}$ satisfies
    \[\langle \nabla u(x)-\nabla u(y),x-y\rangle \geq \left(c_1(|x|^{2r} +|y|^{2r})-c_2(|x|^l+|y|^l)-c_3)\right |x-y|^2 \quad \forall x,y\in \mathbb{R}^d.\]
    where $c_1:=\eta (r+1)$,$c_2=c_3=L$.
    \end{lemma}
    \begin{proof}
        Let $f(x)=|x|^{2r+2}.$ Then $\nabla f(x)= 2(r+1) |x|^{2r} x.$
        Writing \[\begin{aligned}
            \langle \nabla f(x)-\nabla f(y),x-y\rangle&=\langle \nabla f(x),x\rangle +\langle\nabla f(y),y\rangle-\langle \nabla f(x),y\rangle - \langle \nabla f(y),x \rangle\\&= (2r+2) (|x|^{2r+2}+|y|^{2r+2})- (r+1)(|x|^{2r}+|y|^{2r})2\langle x,y\rangle\\&=(2r+2) (|x|^{2r+2}+|y|^{2r+2})\\&+ (r+1)(|x|^{2r}+|y|^{2r})\left( |x-y|^2 -|x|^2-|y|^2\right)
            \\&=(r+1)\left(|x|^{2r+2}+|y|^{2r+2}-|x|^{2r}|y|^{2}-|y|^{2r}|x|^{2}\right) \\&+ (r+1)(|x|^{2r}+|y|^{2r})|x-y|^2.
        \end{aligned}\]
        Since \[\begin{aligned}
            |x|^{2r+2}+|y|^{2r+2}-|x|^{2r}|y|^{2r+2}-|y|^{2r}|x|^{2r+2}&=|x|^{2}(|x|^{2r}-|y|^{2r})-|y|^2(|x|^{2r}-|y|^{2r})\\&=
        (|x|^2-|y|^2)(|x|^{2r}-|y|^{2r})\\&\geq 0,
        \end{aligned}\]
        one deduces
        \begin{equation}
            \langle \nabla f(x)-\nabla f(y),x-y\rangle \geq (r+1)(|x|^{2r}+|y|^{2r})|x-y|^2 \quad \forall x,y \in \mathbb{R}^d.
        \end{equation}
       Noting that by the gradient local Lipschitz assumption on $g$, there holds
        \[\langle \nabla g(x)- \nabla g(y),x-y\rangle\geq -L(1+|x|^l+|y|^l) |x-y|^2 \quad \forall x,y \in \mathbb{R}^d,\] the result immedately follows.
    \end{proof}
    \section{Acknowledgements}
    The authors would like to express their gratidute to Mufan Bill Li for the useful discussions.
\bibliography{referencenew}
\end{document}